\documentclass[letterpaper,10pt]{article}
\usepackage[english]{babel}
\usepackage[utf8x]{inputenc}
\usepackage[T1]{fontenc}
\usepackage[margin=1in]{geometry}
\usepackage{amsmath,amsthm}
\usepackage{graphicx}
\usepackage[colorlinks=true, allcolors=blue]{hyperref}
\usepackage{float}
\usepackage{amsmath,amssymb,amsthm,times,pslatex}
\usepackage[numbers,sort]{natbib}
\usepackage[ruled,vlined,linesnumbered]{algorithm2e}
\usepackage{enumitem}
\usepackage{tikz}
\usetikzlibrary{arrows,decorations.markings}
\usepackage{tabularx}
\usepackage{multirow}
\usepackage{multicol}

\usepackage{tikzit}

\tikzstyle{emptycircle}=[fill=white, draw=black, shape=circle, tikzit shape=circle, inner sep=0pt, minimum size=7mm]
\tikzstyle{output}=[fill=white, draw=black, shape=circle, minimum size=3pt]

\tikzstyle{newstyle}=[draw=black, ->]

\usepackage{array}
\newcolumntype{Z}{>{\centering\let\newline\\\arraybackslash\hspace{0pt}}X}

\usepackage{csquotes}
\usepackage{float}
\usepackage{subcaption}

\DeclareMathOperator{\cb}{Cb}
\DeclareMathOperator{\tp}{tp}
\DeclareMathOperator{\alg}{alg}
\DeclareMathOperator{\ord}{ord}
\DeclareMathOperator{\adim}{adim}
\DeclareMathOperator{\trdeg}{trdeg}

\DeclareMathOperator{\DCF}{DCF}
\DeclareMathOperator{\NEL}{NumExpLoc}
\DeclareMathOperator{\NEG}{NumExpGlob}
\DeclareMathOperator{\df}{defect}
\DeclareMathOperator{\FD}{FD}

\DeclareMathOperator{\I}{\mathcal{I}}
\DeclareMathOperator{\V}{\mathcal{V}}

\DeclareMathOperator{\Q}{\mathbb{Q}}
\DeclareMathOperator{\C}{\mathbb{C}}

\newtheorem{lemma}{Lemma}[section]
\newtheorem{proposition}[lemma]{Proposition}
\newtheorem{corollary}[lemma]{Corollary}
\newtheorem{theorem}[lemma]{Theorem}
\newtheorem{conclusion}[lemma]{Conclusion}

\theoremstyle{definition}

\newtheorem{our_definition}[lemma]{Definition}
\newtheorem{example}[lemma]{Example}
\newtheorem{notation}[lemma]{Notation}
\newtheorem{remark}[lemma]{Remark}

\title{Multi-experiment parameter identifiability of ODEs\\ and model theory}

\author{Alexey Ovchinnikov, Anand Pillay, Gleb Pogudin\footnote{Gleb Pogudin's prior addresses: New York University, Courant Institute of Mathematical Sciences; Higher School of Economics (Moscow), Department of Computer Science}, and Thomas Scanlon \smallskip\\ \small
CUNY Queens College, Department of Mathematics,
65-30 Kissena Blvd, Queens, NY 11367, USA \\ \small
CUNY Graduate Center, Ph.D. Programs and Mathematics and Computer Science, 365 Fifth Avenue,
New York, NY 10016, USA\\ \small
\href{mailto:aovchinnikov@qc.cuny.edu}{aovchinnikov@qc.cuny.edu} \vspace{0.05in} \\ 
 \small University of Notre Dame, Department of Mathematics, Notre Dame, IN 46556, USA \\ \small
\href{mailto:Anand.Pillay.3@nd.edu}{Anand.Pillay.3@nd.edu}\\
 \small LIX, CNRS, \'Ecole Polytechnique, Institute Polytechnique de Paris, Palaiseau, France \\ \small
\href{mailto:gleb.pogudin@polytechnique.edu}{gleb.pogudin@polytechnique.edu}\\ \small University of California, Berkeley, Department of Mathematics, Berkeley, CA 94720-3840, USA\\
\small\href{mailto:scanlon@math.berkeley.edu}{scanlon@math.berkeley.edu}}

\date{}

\begin{document}

\maketitle
\vspace{-0.2in}
\begin{abstract}
Structural identifiability is a property of an ODE model with parameters that allows for the parameters to be determined from continuous noise-free data.
This is a natural prerequisite for practical identifiability.
Conducting multiple independent experiments could make more parameters or functions of parameters identifiable, which is a desirable property to have. 
How many experiments are sufficient? 
In the present paper, we provide an algorithm to determine the exact number of experiments for multi-experiment local identifiability and obtain an upper bound that is off at most by one for the number of experiments for multi-experiment global identifiability.
\\ 
Interestingly, the main theoretical ingredient of the algorithm has been discovered and proved using model theory (in the sense of mathematical logic).
We hope that this unexpected connection will stimulate interactions between applied algebra and model theory, and we provide a short introduction to model theory in the context of parameter identifiability.
As another related application of model theory in this area, we construct a nonlinear ODE system with one output such that single-experiment and multiple-experiment identifiability are different for the system. This contrasts with recent results about single-output linear systems.
\\ 
We also present a Monte Carlo randomized version of the algorithm with a polynomial arithmetic complexity.
Implementation of the algorithm is provided and its performance is demonstrated on several examples.
The source code is available at~\url{https://github.com/pogudingleb/ExperimentsBound}.
\end{abstract}

\section{Introduction}

Structural identifiability  is a property of an ODE system with parameters that allows for the parameters to be uniquely  determined (global identifiability) or determined up to finitely many choices (local identifiability) from noiseless data and sufficiently exciting inputs (also known as the persistence of excitation, see~\cite{LG94, V18, XiaMoog}). 
Performing  structural identifiablity analysis is an important first step in evaluating and, if needed, adjusting the system before a reliable practical parameter identification is performed. 

For an ODE model, some of the parameters or functions of parameters could be non-identifiable from a single experiment, but these parameters could become identifiable if one conducts more than one experiment~\cite{VECB, allident}.
Knowing the number of experiments to be conducted to achieve the maximal possible identifiability (that is, if the  identifiability did not occur after this many experiments, it will not occur after additional experiments) is important for designing experimental protocols involving several experiments~\cite[Section III.B]{VECB}.
In particular, making this number smaller would allow for less expensive experimental protocols.
Also, knowing these bounds, one can use existing software for assessing local~\cite{Sedoglavic, KAJ2012, STRIKE_GOLDD} or global~\cite{sian} single-experiment identifiability to check multi-experiment identifiability of parameters (or functions of parameters) of interest.
Note that, due to~\cite[Theorem~19]{allident}, one can alternatively use software based on input-output equations~\cite{DAISY_IFAC, COMBOS} to find the multi-experiment identifiable functions, but this approach does not determine
 the number of experiments.

One can find such a number of experiments for the case of local identifiability using the algorithm presented in~\cite{VECB} if all parameters are locally identifiable.
In~\cite[Section 4]{allident}, we gave an algorithm computing, among other things, an upper bound for the number of experiments to achieve the maximal possible global identifiability. 
The proposed algorithm had two drawbacks: it used the Rosenfeld-Gr\"obner algorithm for differential elimination, which may be computationally very expensive, and the resulting bound could be arbitrarily far from the exact one.

In this paper, we 
present an algorithm that computes:
\begin{itemize}
    \item the smallest number of experiments to achieve the maximal possible local identifiability,
    \item the number of experiment to achieve the maximal possible global identifiability so that this number exceeds the minimal such number by at most one.
\end{itemize}
We present a randomized Monte Carlo version of this algorithm having polynomial arithmetic complexity (see Section~\ref{sec:alg_theory}).
We have implemented this algorithm in Julia language, demonstrated its performance on several examples, and compared with the algorithm from~\cite{allident} (see~Section~\ref{sec:impl_ex}).

Our algorithm is based on theoretical properties of multi-experiment identifiability  that we establish (summarized in Section~\ref{sec:main_results}).
The process of discovery and establishing of these properties originated from model theory (in the sense of mathematical logic).
The use of model theory in this area is novel, so we give a brief overview here. Already differential algebra plays a large role in structural identifiability (see, e.g., \cite{LG94}).
Differential algebra and the study of solution sets of ODEs in differential rings and fields, are enhanced by model-theoretic perspectives and methods, especially from \emph{stability theory}.  In particular, differential fields of definition (of differential ideals) are special cases of \emph{canonical bases} from model theory. Using model-theoretic ideas in a non-trivial way, we will prove new quantitative results on recovering such differential fields of definitions from  sufficiently many \emph{independent solutions}. 
In Section~\ref{sec:model_theory}, we will elaborate on the relationship between the setup of multi-experiment  
identifiability and that of differential algebra/model theory, where we will give additional references.

We also use model theoretic tools to construct an example contrasting with recent results about multi-experiment identifiability of linear systems~\cite{OPT19}.
\cite[Theorem~1]{OPT19} implies that single-experiment and multi-experiment identifiability are the same thing for linear ODE systems with one output. 
In Section~\ref{sec:counterexample}, we show a series of non-linear systems for which this is not the case.

The paper is organized as follows.
In Section~\ref{sec:preliminaries}, we give basic definitions from differential algebra and structural identifiability.
Section~\ref{sec:main_results} summarizes our main results.
Section~\ref{sec:counterexample} contains the example of a single-output system for which single-experiment and multi-experiment identifiability do not coincide.
 In section~\ref{sec:proofs}, we give an algebraic proof of Theorem~\ref{thm:num_exp}, which is the main theoretical ingredient of our algorithm.
In Section~\ref{sec:alg}, we present our algorithm, analyze its complexity, describe our implementation, and demonstrate it on a set of examples (including comparison with the algorithm from~\cite{allident}).
 In Section~\ref{sec:model_theory}, we describe the connections between identifiability and model theory and explain the model-theoretic context of the results in this paper.
 The section is aimed at readers who are interested in knowing why and how model-theoretic methods are useful in or relevant to identifiability problems. Although the reader need not be a specialist in model theory, they should either be acquainted with the basic notions or be willing to follow up the references.

Our implementation together with all  examples used in the paper can be found at~\url{https://github.com/pogudingleb/ExperimentsBound}.


\section{Preliminaries}\label{sec:preliminaries}

\subsection{Differential algebra}
\begin{our_definition}[Differential rings and fields]
\begin{itemize}
\item[]
  \item A {\em differential ring} $(R,\,')$ is a commutative ring with a derivation $':R\to R$, that is, a map such that, for all $a,b\in R$, $(a+b)' = a' + b'$ and $(ab)' = a' b + a b'$. 
  A differential ring that is also a field is called \emph{a differential field}.
  
  \item For an extension of differential fields $F \subset E$ and elements $a_1, \ldots, a_s \in E$, let $F\langle a_1, \ldots, a_n\rangle$  denote the smallest differential subfield of $E$ containing $F$ and $a_1, \ldots, a_n$.
  
\end{itemize}
\end{our_definition}

\begin{our_definition}[Differential polynomials and differential ideals]
  \begin{itemize}
  \item[]
      \item The {\em ring of differential polynomials} in the variables $x_1,\ldots,x_n$ over a field $K$ is the ring $K[x_j^{(i)}\mid i\geqslant 0,\, 1\leqslant j\leqslant n]$ with a derivation defined on the ring by $(x_j^{(i)})' := x_j^{(i + 1)}$. 
  This differential ring is denoted by $K\{x_1,\ldots,x_n\}$.
     \item An ideal $I$ of a differential ring $(R,\ ')$ is called a {\em differential ideal} if, for all $a \in I$, $a'\in I$. For $F\subset R$, the smallest differential ideal containing set $F$ is denoted by $[F]$.
    \item For an ideal $I$ and element $a$ in a ring $R$, we denote $I \colon a^\infty = \{r \in R \mid \exists \ell\colon a^\ell r \in I\}$.
     This set is  an ideal in $R$.
  \end{itemize}
\end{our_definition}


\subsection{Identifiability}

We will consider \emph{an algebraic differential model}
\begin{equation}\label{eq:sigma}
    \Sigma := \begin{cases}
      \bar{x}' = \bar{f}(\bar{x}, \bar{\mu}, \bar{u}),\\
      \bar{y} = \bar{g}(\bar{x}, \bar{\mu}, \bar{u}),
    \end{cases}
\end{equation}
where 
\begin{itemize}
    \item $\bar{f} = (f_1, \ldots, f_n)$ and $\bar{g} = (g_1, \ldots, g_m)$ are tuples of rational functions 
    over
    $\mathbb{C}$;
    \item $\bar{x}$, $\bar{u}$, $\bar{y}$ are state, input, and output variables, respectively;
    \item $\bar{\mu} = (\mu_1, \ldots, \mu_\ell)$ are parameters.
\end{itemize}
The analytic notion of identifiability~\cite[Definition~2.5]{HOPY2020} is equivalent~(see \cite[Proposition~3.4]{HOPY2020} and~\cite[Proposition~4.7]{OPT19}) to the following algebraic definition, which we will use.

We write $\bar{f} = \frac{\bar{F}}{Q}$ and $\bar{g} = \frac{\bar{G}}{Q}$, where $\bar{F}$ and $\bar{G}$ are tuples of polynomials over $\mathbb{C}(\bar{\mu})$ 
and $Q$ is the common denominator of $\bar{f}$ and $\bar{g}$. Here we consider $\C(\bar{\mu})$ as a differential field of constants.
Then we define a differential ideal
\begin{equation}\label{eq:Isigma}
    I_{\Sigma} := [Qx_1' - F_1, \ldots, Qx_n' - F_n, Qy_1 - G_1, \ldots, Qy_m - G_m] \colon Q^\infty \subset \mathbb{C}(\bar{\mu})\{\bar{x}, \bar{y}, \bar{u}\}.
\end{equation}
Observe that every solution of~\eqref{eq:sigma} is a solution of $I_\Sigma$.

\begin{our_definition}[Generic solution]\label{def:generic_solution}
  A tuple $(\bar{x}^\ast, \bar{y}^\ast, \bar{u}^\ast)$ from a differential field $k \supset \mathbb{C}(\bar{\mu})$ is called 
  \emph{a generic solution} of~\eqref{eq:sigma} if, for every differential polynomial $P \in \mathbb{C}(\bar{\mu})\{\bar{x}, \bar{y}, \bar{u}\}$, 
  \[
    P(\bar{x}^\ast, \bar{y}^\ast, \bar{u}^\ast) = 0 \iff P \in I_{\Sigma}.
  \]
\end{our_definition}

\begin{remark}
  \cite[Lemma~3.2]{HOPY2020} implies that $I_{\Sigma}$ is a prime differential ideal. 
  Therefore, it has a generic solution.
\end{remark}

\begin{our_definition}[Identifiability: single-experimental]\label{def:ident_si}
  For a model $\Sigma$ in~\eqref{eq:sigma}, a rational function $h \in \mathbb{C}(\bar{\mu})$ is said to be \emph{globally (resp., locally) single-experiment identifiable} (SE-identifiable) if, for every generic solution $ ( \bar{x}^\ast,      \bar{y}^\ast, \bar{u}^\ast)$ of~\eqref{eq:sigma} considered as a system of differential equations over $\mathbb{C}(\bar{\mu})$, we have
  \[
    h(\bar{\mu}) \in \mathbb{C}\langle \bar{y}^\ast, \bar{u}^\ast \rangle\quad\quad \text{(resp., $h(\bar{\mu})$ is algebraic over $\mathbb{C}\langle \bar{y}^\ast, \bar{u}^\ast \rangle$)}.
  \]
\end{our_definition}

\begin{remark}
  The equivalence of Definition~\ref{def:ident_si} to the more common analytic definition of identifiability has been established in~\cite[Propsition~3.4]{HOPY2020}.  For a comparison with other definitions, see \cite[Section~2.1.1]{ADM2020}.
\end{remark}

\begin{our_definition}[Identifiability defect]\label{def:defect}
  For a model $\Sigma$ in~\eqref{eq:sigma}, we define the \emph{identifiability defect} as
  \[
    \df(\Sigma) := \trdeg_{\C\langle \bar{y}^\ast, \bar{u}^\ast \rangle} \C(\bar{\mu})\langle \bar{y}^\ast, \bar{u}^\ast \rangle,
  \]
  where $(\bar{x}^\ast, \bar{y}^\ast, \bar{u}^\ast)$ is any generic solution of~\eqref{eq:sigma} (one can show that the defect does not depend on the choice of the generic solution).
  
  For example, $\df(\Sigma) = 0$ implies that all the parameters are locally identifiable.
\end{our_definition}

\begin{our_definition}[{Identifiability: multi-experimental, \cite[Definition~16]{allident}}]\label{def:ident:multi}
  \begin{itemize}
  \item[]
      \item For a model $\Sigma$ and a positive integer $r$, we define \emph{the $r$-fold replica} of $\Sigma$ as
      \[
      \Sigma_r := \begin{cases}
          \bar{x}_i' = \bar{f}(\bar{x}_i, \bar{\mu}, \bar{u}_i), \;\; i = 1, \ldots, r,\\
          \bar{y}_i = \bar{g}(\bar{x}_i, \bar{\mu}, \bar{u}_i), \;\; i = 1, \ldots, r,
      \end{cases}
      \]
      where $\bar{x}_1, \ldots, \bar{x}_r, \bar{y}_1, \ldots, \bar{y}_r, \bar{u}_1, \ldots, \bar{u}_r$ are new tuples of indeterminates (note that the vector of parameters is not being replicated).
      
      \item For a model $\Sigma$, a rational function $h \in \mathbb{C}(\bar{\mu})$ is called \emph{globally (resp., locally) multi-experimental identifiable} (ME-identifiable) if there exists a positive integer $r$ such that $h(\bar{\mu})$ is globally (resp., locally) SE-identifiable in $\Sigma_r$.
  \end{itemize}
\end{our_definition}

\begin{notation}
   For a differential model $\Sigma$ we define $\mathrm{NumExpGlob}(\Sigma)$ (resp., $\mathrm{NumExpLoc}(\Sigma)$) as the smallest integer $r$ such that, for every $h(\bar{\mu}) \in \mathbb{C}(\bar{\mu})$, the following are equivalent:
   \begin{itemize}
     \item $h$ is globally (resp., locally) multi-experimental identifiable for $\Sigma$;
     \item $h$ is globally (resp., locally) single-experimental identifiable for $\Sigma_r$.
   \end{itemize}
\end{notation}

\begin{remark}
A simple family of linear models $\Sigma_r$ that reaches arbitrarily high values for the number of experiments, $\NEL(\Sigma_r)=\NEG(\Sigma_r)=r$, is \cite[Example~30]{allident}.
\end{remark}



\section{Main results}
\label{sec:main_results}

Our results consist of a theoretical part and algorithms building upon the theory.
The theoretical contribution is summarized in the statements below.
In particular, Theorem~\ref{thm:num_exp} is the basis for an algorithm for computing $\NEL(\Sigma)$ and obtaining an upper bound for $\NEG(\Sigma)$ that is off at  most by one.
 \cite[Theorem~4.2]{OPT19} implies that, for single-output linear models, identifiable and multi-experiment identifiable functions coincide. Conclusion~\ref{thm:example} indicates that this  does not generalize to nonlinear models.

\begin{theorem}\label{thm:num_exp}
  For every algebraic differential model $\Sigma$ of the form~\eqref{eq:sigma}, we have:\vspace{0.2cm}
  \begin{enumerate}
      \item $\NEL(\Sigma) = \min \{r \mid \df(\Sigma_r) = \df(\Sigma_{r + 1}) \}$.
      \item $\NEL(\Sigma) \leqslant \NEG(\Sigma) \leqslant \NEL(\Sigma) + 1$.
  \end{enumerate}
\end{theorem}

\begin{corollary}
  For every algebraic differential model $\Sigma$ of the form~\eqref{eq:sigma} with $\ell$ parameters:
  \[
    \NEL(\Sigma) \leqslant \ell\quad\text{ and }\quad\NEG(\Sigma) \leqslant \ell + 1.
  \]
\end{corollary}

\begin{conclusion}\label{thm:example}
  There is a system of the form~\eqref{eq:sigma} with a single output such that the fields of identifiable and multi-experiment identifiable functions do not coincide.
\end{conclusion}

Proofs of Theorem~\ref{thm:num_exp} and Conclusion~\ref{thm:example} are presented in Sections~\ref{sec:proofs} and~\ref{sec:counterexample}, respectively.

On the algorithmic side, Theorem~\ref{thm:num_exp} yields a probabilistic algorithm for computing the value $\NEL(\Sigma)$ and a bound for $\NEG(\Sigma)$ that
is off at most by one with arithmetic complexity being polynomial in the complexity of the system (Proposition~\ref{prop:complexity}).
We implemented this algorithm, and we demonstrate its practical performance and apply it to examples in Section~\ref{sec:alg}.

\section{Single-output model requiring more than one experiment}\label{sec:counterexample}

In this section, we will prove Conclusion~\ref{thm:example} by showing that the SE-identifiable and ME-identifiable functions do not coincide for the following model $\Sigma$:
\begin{equation}\label{eq:counterexsys}
    \begin{cases}
           x_1' = 0,\\
           x_2' = x_1x_2 + \mu_1 x_1 + \mu_2,\\
           y = x_2.
    \end{cases}
\end{equation}
We will  now
give a direct algebraic proof. In Section~\ref{sec:cemodeltheory}, we present the model-theoretic argument that has been used to construct this example and can be used to construct more complex ones.

\begin{lemma}
  The field of ME-identifiable function of~\eqref{eq:counterexsys} is $\C(\mu_1, \mu_2)$ but neither 
  $\mu_1$ 
   nor 
  $\mu_2$ 
  is
  SE-identifiable.
\end{lemma}

\begin{proof}
 We 
  find the field of ME-identifiable function using~\cite[Theorem~19]{allident}.
  Differentiating the second equation in~\eqref{eq:counterexsys}, we get $x_2'' = x_1x_2'$.
  Using this equation, we can eliminate $x_1$ from the second equation of~\eqref{eq:counterexsys} and obtain:
  \begin{equation}\label{eq:counter_input_output}
  x_2x_2'' - (x_2')^2 + \mu_1x_2'' + \mu_2 x_2' = 0.
  \end{equation}
  Since $x_2$ does not satisfy any first order equation over $\C(\mu_1, \mu_2)$ modulo~$I_\Sigma$ and \eqref{eq:counter_input_output} is irreducible,  the set consisting of  \eqref{eq:counter_input_output} is a set of input-output equations for $\Sigma$, so~\cite[Theorem~19]{allident} implies that the coefficients $\mu_1$ and $\mu_2$  of \eqref{eq:counter_input_output}
are ME-indentifiable. 
  
  To prove that $\mu_1$ and $\mu_2$ are not SE-identifiable, consider a generic solution $(x_1^\ast, x_2^\ast, y^\ast)$ of~\eqref{eq:counterexsys}.
  Then there is a differential automorphism of $\C(\bar{\mu})\langle x_1^\ast, x_2^\ast, y^\ast\rangle$ defined by
  \[
  \alpha|_{\C\langle x_1^\ast, x_2^\ast, y^\ast\rangle} = \operatorname{id}, \quad \alpha(\mu_1) = \mu_1 + 1, \quad \alpha(\mu_2) = \mu_2 - x_1.
  \]
  Therefore, neither of $\mu_1$ or $\mu_2$ belongs to $\C\langle y^\ast \rangle$.
\end{proof}

\begin{remark}
  Using~\cite[Algorithm~1]{allident}, one can show that the field of SE-identifiable functions of~\eqref{eq:counterexsys} is $\C$.
\end{remark}


\section{Bounding the number of experiments (proof of Theorem~\ref{thm:num_exp})}\label{sec:proofs}

In this subsection, we will give an algebraic proof (but with a strong model theoretic flavor, which we expand in Section~\ref{sec:mtpf}) of Theorem~\ref{thm:num_exp}.
We start with fixing some notation for the subsection.

\begin{notation}\label{not:tuples}
\begin{itemize}
    \item[]
    \item $\bar{a}, \bar{b}$, and $\bar{c}$ denote tuples of elements of differential fields.
    \item $\bar{x}, \bar{y}$, and $\bar{z}$ denote tuples of differential indeterminates. 
Moreover, we will assume that $|\bar{a}| = |\bar{x}|$, $|\bar{b}| = |\bar{y}|$, and $|\bar{c}| = |\bar{z}|$.
     
    \item $k_0$ will be a fixed differential ground  field (in applications, $k_0 = \mathbb{Q}, \mathbb{R}, \mathbb{C}$ with zero derivation).
    We will also consider an extension $K \supset k_0$ such that $K$ is differentially closed and $|k_0|$-saturated field.
    Saturation and differentially closed fields are defined in Section~\ref{sec:model_theory}, in this section we will use only the following algebraic consequence of these properties~\cite[Propositions~4.2.13 and~4.3.3 and page~117]{Marker_book}: for every subfield $k \subset K$ of cardinality at most $|k_0|$ and every differential automorphism  $\alpha$ of $k$, $\alpha$ can be extended to an endomorphism of $K$.

    \item Let $E$ be a field and $\bar{a}$ be a tuple of elements from some extension of $E$.
    Then $\trdeg_E \bar{a}$ denotes $\trdeg_E E(\bar{a})$.
\end{itemize}
\end{notation}

\begin{notation}\label{not:fields}
   Let $k \subset K$ be an intermediate differential field (in applications, we will have $k = \C(\bar{\mu})$)
   and $\bar{a}$ a tuple from $K$.
  \begin{itemize}
      \item 
      \emph{The vanishing ideal of $\bar{a}$ over $k$} is denoted by
      \[
        \I_k(\bar{a}) := \{ p \in k\{\bar{x}\} \mid p(\bar{a}) = 0 \}.
      \]
      \item We denote \emph{the differential-algebraic variety of $\bar{a}$ with respect to $K$ defined over $k$} by 
      \[
        \V_{K/k}(\bar{a}) = \{ \bar{b} \in K \mid p(\bar{b}) = 0 \quad \forall p \in \I_k(\bar{a}) \} \subset K^{|\bar{a}|}.
      \]
     We consider this as a differential-algebraic variety over $K$, and it is not necessarily irreducible.
     For brevity, until the end of the section, by ``variety'' we will mean ``differential-algebraic variety''.
      \item For a tuple $\bar{a}$ from $K$, $\FD_{k}(\bar{a})$ denotes the field generated by $k_0$ and the field of definition of $\mathcal{I}_k(\bar{a})$  (cf. Example~\ref{ex:ident_me}).
      
      \item 
      Let $\langle \mathcal{I}_k(\bar{a}), \mathcal{I}_k(\bar{b}) \rangle$  denote the ideal in $k\{\bar{x}, \bar{y}\}$ generated 
      by
      $\mathcal{I}_k(\bar{a}) \subset k\{\bar{x}\}$ and $\mathcal{I}_k(\bar{b}) \subset k\{\bar{y}\}$.
      For tuples $\bar{a}_1, \ldots, \bar{a}_n$, the ideal $\langle \mathcal{I}_k(\bar{a}_1), \ldots, \mathcal{I}_k(\bar{a}_n) \rangle$ is defined analogously.
      
      \item For tuples $\bar{a}_1$ and $\bar{a}_2$ of the same length, we write $\I_k(\bar{a}_1) \cong \I_k(\bar{a}_2)$ if the ideals $\I_k(\bar{a}_1)$ and $\I_k(\bar{a}_2)$ coincide if being considering in the same ring $k\{\bar{x}\}$.
      
  \end{itemize}
\end{notation}


\begin{lemma}\label{lem:independence}
  Let $k \subset K$ be a differential subfield and $\bar{a}$ and $\bar{b}$  tuples from $K$ 
  such that $\I_k(\bar{a}, \bar{b}) = \langle \I_k(\bar{a}), \I_k(\bar{b}) \rangle$.
  Then $\I_{k\langle \bar{b} \rangle} (\bar{a})$ is generated by $\I_k(\bar{a})$. 
\end{lemma}

\begin{proof}
  By clearing denominators, every element $p \in \I_{k\langle \bar{b} \rangle} (\bar{a})$ can be written as $p = 
  q(\bar{x}, \bar{b})/d$, where $d \in k\langle \bar{b} \rangle$ and $q(\bar{x}, \bar{y}) \in \I_k(\bar{a}, \bar{b})$.
The differential  polynomial $q(\bar{x}, \bar{y})$ can be written as a combination of elements of $\I_k(\bar{a})$ and $\I_k(\bar{b})$.
  If we plug $\bar{y} = \bar{b}$, the terms from $\I_k(\bar{b})$ will vanish, so $q(\bar{x}, \bar{b})$ can be written as a combination of elements of $\I_k(\bar{a})$.
  Then the same is true for $p$.
\end{proof}

\begin{proposition}[{cf.~\cite[Theorem~19]{allident}}]\label{prop:existence_N}
  Let $k \subset K$ be a differential subfield and $\bar{a}_1, \bar{a}_2, \ldots$ be tuples of the same length from $K$ with
  \[
    \I_k(\bar{a}_1) \cong \I_k(\bar{a}_2) \cong \ldots \quad \text{and} \quad \I_k(\bar{a}_1, \ldots, \bar{a}_\ell) = \langle \I_k(\bar{a}_1), \ldots, \I_k(\bar{a}_\ell) \rangle \text{ for every }\ell \geqslant 1.
  \]
  Then there exists $N$ such that (see Notation~\ref{not:fields})
  \[
    \FD_k(\bar{a}_1) \subset k_0\langle \bar{a}_1, \ldots, \bar{a}_N \rangle.
  \]
\end{proposition}

\begin{proof}
  Let $L = k_0\langle \bar{c} \rangle$ be generated by the field of definition of $J := \I_k(\bar{a}_1)$ and  by $k_0$.
  We consider an arbitrary ordering of the monomials of the corresponding differential ring, and consider a linear basis of $J$ that is in the reduced row echelon form with respect to this ordering (this construction is described in more details in the proof of~\cite[Theorem~4.7]{ident-compare}).
  The elements of this basis form a set of generators $\{f_\lambda\}_{\lambda \in \Lambda}$ of $J$ such that, for every $\lambda \in \Lambda$,
  \begin{itemize}
      \item the coefficients of $f_\lambda$ are in $L$ and at least one of them is $1$;
      \item for every $g \in J \setminus \{f_\lambda\}$, the support of $f_\lambda - g$ is not a proper subset of the support of $f_\lambda$.
  \end{itemize}
  
  Then the coefficients of $\{f_{\lambda}\}_{\lambda \in \Lambda}$ generate $L$ over $k_0$.
  We fix some $\lambda$ and write $f_{\lambda} = m_0 + c_1m_1 + \ldots + c_Nm_N$, where $m_1, \ldots, m_N$ are differential monomials and $c_1, \ldots, c_N \in L$.
  We will show that $c_1, \ldots, c_N \in \Q\langle \bar{a}_1, \ldots, \bar{a}_N\rangle$.
  For every $j \geqslant 1$, we denote the monomial $m_i$ evaluated at $\bar{a}_j$ by $m_{j, i}$.
  Then we have a linear system in $c_1, \ldots, c_N$
  \[
  \begin{pmatrix}
    m_{1, 1} & \ldots & m_{1, N} \\
    \vdots & \ddots & \vdots \\
    m_{N, 1} & \ldots & m_{N, N}
  \end{pmatrix}
  \begin{pmatrix}
    c_1\\
    \vdots\\
    c_N
  \end{pmatrix}
  =
  \begin{pmatrix}
    -m_{1, 0}\\
    \vdots\\
    -m_{N, 0}
  \end{pmatrix}.
  \]
  We denote the matrix of the above system by $S$.
  If $\det S \neq 0$, then, solving the system, 
  we show that $c_1, \ldots, c_N \in \Q\langle \bar{a}_1, \ldots, \bar{a}_N \rangle$.
  Assume that $\det S = 0$ and
  let $S_0$ be the smallest singular minor of $S$.
  Let the index of the first row of $S_0$ be $s$, and define $E := k\langle \bar{a}_1, \ldots, \bar{a}_{s - 1}, \bar{a}_{s + 1}, \ldots, \bar{a}_N\rangle$.
  Expanding the equality $\det S_0 = 0$ with respect to the first row, we obtain an element $p$ of $\I_{E}(\bar{a}_s)$.
  The minimality of $S_0$ implies that $p \neq 0$.
  Since the support of $p$ is a proper subset of the support of $f_{\lambda}$, we conclude that $\I_{E}(\bar{a}_s)$ is not generated by $\I_{k}(\bar{a}_s)$.
  This contradicts Lemma~\ref{lem:independence}.
  
  Since $L$ is finitely generated (follows, for example, from~\cite[Proposition~4.11]{OPT19}), the coefficients of finitely many $f_{\lambda}$'s generate $L$ over $k$.
  Taking the maximum of the corresponding $N$'s will give the desired $N$ and finish the proof.
\end{proof}


\begin{lemma}\label{lem:components_are_generic}
  Let $k$ with $k_0 \subset k \subset K$ be a differential field differentially finitely generated over $k_0$ and $\bar{a}$ a tuple from $K$.
  Then, for every component $C$ of $\V_{K/k}(\bar{a})$, there exists $\bar{b} \in C$ such that $\I_{k_0}(\bar{a}) \cong \I_{k_0}(\bar{b})$.
\end{lemma}

\begin{proof}
  The ideal $\I_k(\bar{a})$ is prime.
  Then the Galois group of $k^{\alg} \subset K$ over $k$ acts transitively on the components of $\V_{K/k}(\bar{a})$.
  Let $C_0$ be a component containing $\bar{a}$, and let $\alpha$ be an automorphism of $k^{\alg}$ over $k$ that maps $C_0$ to $C$.
  By Notation~\ref{not:tuples}, since $|k| = |k_0|$, $\alpha$ can be lifted to a differential endomorphism of $K$ which we will denote by $\alpha$ as well.
  We set $\bar{b} := \alpha(\bar{a})$.
  Then we have $\I_{k_0}(\bar{a}) \cong \I_{k_0}(\bar{b})$ due to the $\alpha$-invariance of $k_0$.
\end{proof}


\begin{lemma}\label{lem:almost_intersection}
  Let $k$ with $k_0 \subset k \subset K$ be a differential subfield.
  Let 
  \begin{itemize}
      \item $\I_k(\bar{a}, \bar{b}) = \langle\I_k(\bar{a}), \I_k(\bar{b}) \rangle$;
      \item $\bar{c}$ from $k$ be such that $\FD_k(\bar{a}) \cup \FD_k(\bar{b}) \subset k_0\langle \bar{c} \rangle$ (see Notation~\ref{not:fields});
      \item a component $C \subset \V_{K/k_0\langle \bar{a} \rangle}(\bar{c})$ be such that $C \subset \V_{K/k_0\langle \bar{b} \rangle}(\bar{c})$.
  \end{itemize}
  Then $C$ is a component of $\V_{K/k_0\langle \bar{a}, \bar{b} \rangle}(\bar{c})$.
\end{lemma}

\begin{proof}
  Consider any $p \in \I_{k_0\langle \bar{a}, \bar{b} \rangle}(\bar{c}) \subset k_0\langle\bar{a}, \bar{b} \rangle \{\bar{z}\}$.
  Let $d \in k_0\{\bar{a}, \bar{b}\}$ be the product of the denominators of the coefficients of $p$.
  Then there exists $q \in k_0\{\bar{x}, \bar{y}, \bar{z}\}$ such that $q(\bar{a}, \bar{b}, \bar{z}) = dp$.
  Since $q(\bar{a}, \bar{b}, \bar{c}) = 0$, we have
  \[
  q(\bar{x}, \bar{y}, \bar{c}) \in \I_{k_0\langle \bar{c} \rangle}(\bar{a}, \bar{b}) = k_0\langle \bar{c} \rangle \{\bar{x}, \bar{y}\}\I_{k_0\langle \bar{c} \rangle}(\bar{a}) + k_0\langle \bar{c}\rangle \{\bar{x}, \bar{y}\}\I_{k_0\langle \bar{c} \rangle}(\bar{b}),
  \]
  where the latter equality follows from 
  $\I_k(\bar{a}, \bar{b}) = \langle\I_k(\bar{a}), \I_k(\bar{b}) \rangle$
  and the fact that $k_0\langle \bar{c}\rangle$ contains the fields of definitions of $\I_k(\bar{a})$ and $\I_k(\bar{b})$.
  By clearing the denominators with respect to $\bar{c}$, we conclude that there exists $h(\bar{z}) \in k_0\{\bar{z}\}$ such that $h(\bar{c}) \neq 0$ and
  \[
    h(\bar{z}) q(\bar{x}, \bar{y}, \bar{z}) \in k_0\{\bar{x}, \bar{y}, \bar{z}\}\I_{k_0}(\bar{a}, \bar{c}) + k_0 \{\bar{x}, \bar{y}, \bar{z}\}\I_{k_0}(\bar{b}, \bar{c}).
  \]
  Thus, $h(\bar{z}) q(\bar{a}, \bar{b}, \bar{z})$ vanishes on $\V_{K/k_0\langle \bar{a} \rangle}(\bar{c}) \cap \V_{K/k_0\langle \bar{b} \rangle}(\bar{c})$ and, consequently, on $C$.
  If $q(\bar{a}, \bar{b}, \bar{z})$ does not vanish on $C$,
  then $h(\bar{z})$ does.
  However, this is impossible due to Lemma~\ref{lem:components_are_generic} because $h(\bar{z}) \not\in \I_{k_0}(\bar{c})$.
  
  Thus, $p$ vanishes on $C$, and so $C \subset \V_{K/k_0\langle \bar{a}, \bar{b} \rangle}(\bar{c})$.
  On the other hand, $\V_{K/k_0\langle \bar{a}, \bar{b} \rangle}(\bar{c}) \subset \V_{K/k_0\langle \bar{a}\rangle}(\bar{c})$, so $C$ is a component of $\V_{K/k_0\langle \bar{a}, \bar{b} \rangle}(\bar{c})$.
\end{proof}


\begin{lemma}\label{lem:map_independent}
  Let $k$ with $k_0 \subset k \subset K$ be a differential field differentially finitely generated over $k_0$.
  Consider tuples $\bar{a}_1, \ldots, \bar{a}_\ell$ of the same length from $K$ 
  such that
  \[
    \I_k(\bar{a}_1) \cong \ldots \cong \I_k(\bar{a}_\ell) \quad \text{ and }\quad \I_k(\bar{a}_1, \ldots, \bar{a}_\ell) = \langle \I_k(\bar{a}_1), \ldots, \I_k(\bar{a}_\ell) \rangle.
  \]
  Then, for every permutation $\pi \in S_\ell$, there exists an endomorphism $\alpha_\pi$ of $K$ over $k$ such that $\alpha_{\pi}(\bar{a}_i) = \bar{a}_{\pi(i)}$ for every $1 \leqslant i \leqslant \ell$.
\end{lemma}

\begin{proof}
  Consider the ideal $J = \I_k(\bar{a}_1, \ldots, \bar{a}_\ell) \subset k\{\bar{x}_1, \ldots, \bar{x}_\ell\}$.
  We fix $\pi \in S_\ell$.
  Let $\beta_\pi$ be the differential $k$-automorphism of $k\{\bar{x}_1, \ldots, \bar{x}_\ell\}$ defined by $\beta_{\pi}(\bar{x}_i) = \bar{x}_{\pi(i)}$ for every $1 \leqslant i \leqslant \ell$.
  Since the set $\{\I_k(\bar{a}_1), \ldots, \I_k(\bar{a}_\ell)\}$, where each $I_k(\bar{a}_j)$ is considered as a subset of $k\{\bar{x}_1, \ldots, \bar{x}_\ell\}$, is $\beta_{\pi}$-invariant, so is~$J$.
  Therefore, $\beta_{\pi}$ yields an automorphism, say $\alpha_{\pi}$, of
  \[
  k\{\bar{x}_1, \ldots, \bar{x}_\ell\} / J \cong k\{\bar{a}_1, \ldots, \bar{a}_s\}.
  \]
  $\alpha_\pi$ can be lifted uniquely to an automorpism of $k\langle \bar{a}_1, \ldots, \bar{a}_s\rangle$.
  The resulting automorphism can be lifted to an endomorphism of $K$ by Notation~\ref{not:tuples} since $|k\langle \bar{a}_1, \ldots, \bar{a}_s\rangle| = |k_0|$.
\end{proof}


\begin{lemma}\label{lem:trdeg}
  Let $k$ with $k_0 \subset k \subset K$ be a differential field and $\bar{a}$ a tuple from $K$ and let $F := \FD_k(\bar{a})$ (see Notation~\ref{not:fields}) be such that $\trdeg_{F} k < \infty$. 
  Then $\trdeg_{F\langle \bar{a} \rangle} k = \trdeg_{F} k$.
\end{lemma}

\begin{proof}
  Let $\alpha_1, \ldots, \alpha_N$ be a transcendence basis of $k$ over $F$.
  Assume that $\alpha_1, \ldots, \alpha_N$ are algebraically dependent over $F\langle \bar{a}\rangle$.
  Then there exists  $P \in F \{\bar{x}\}[y_1, \ldots, y_N]$ such that $P(\bar{a}, \alpha_1, \ldots, \alpha_n) = 0$ and $P(\bar{a}, y_1, \ldots, y_n) \neq 0$.
  On the other hand, since the field of definition of $\mathcal{I}_k(\bar{a})$ is $F$, and the monomials in $\alpha_1, \ldots, \alpha_N$ are $F$-linearly independent, every coefficient of $P$ as a polynomial in $y_1, \ldots, y_N$ vanishes at $\bar{a}$.
  Thus, $P(\bar{a}, y_1, \ldots, y_N) = 0$. Contradiction.
\end{proof}


\begin{notation}
  For an irreducible differential-algebraic variety $X \subset K^n$, let $\adim X$ denote the \emph{algebraic dimension}, that is the transcendence degree of the algebra of regular functions.
  The algebraic dimension of an arbitrary differential-algebraic variety is defined as the maximum of the algebraic dimensions of  
  its components.
\end{notation}


\begin{proposition}\label{prop:bound_N}
  Let:
  \begin{itemize}
  \item $k$ with $k_0 \subset k \subset K$ be a differential field of finite transcendence degree over $k_0$,
  \item 
$\bar{a}_1, \bar{a}_2, \ldots$ be tuples of the same length from $K$  
   such that
  \[
    \I_k(\bar{a}_1) \cong \I_k(\bar{a}_2) \cong \ldots \quad \text{and} \quad \I_k(\bar{a}_1, \ldots, \bar{a}_\ell) = \langle \I_k(\bar{a}_1), \ldots, \I_k(\bar{a}_\ell) \rangle \text{ for every }\ell \geqslant 1,
  \]
  \item 
  $r$ the smallest integer such that 
$
  \trdeg_{k_0\langle \bar{a}_1, \ldots, \bar{a}_{r} \rangle} k = \trdeg_{k_0\langle \bar{a}_1, \ldots, \bar{a}_{r + 1} \rangle} k
$.
  \end{itemize}
  Then 
  \begin{enumerate}[label=(\arabic*)]
      \item $r$ is the smallest integer such that $\FD_k(\bar{a}_1)$ (see Notation~\ref{not:fields}) is algebraic over $k_0\langle \bar{a}_1, \ldots, \bar{a}_r \rangle$;
      \item $\FD_k(\bar{a}_1) \subset k_0\langle \bar{a}_1, \ldots, \bar{a}_{r + 1}\rangle$.
  \end{enumerate}
\end{proposition}

\begin{proof}
  Let $\mathcal{C}$ be the field of constants of $K$ and $\bar{c}$ be any set of generators of $\FD_k(\bar{a}_1)$.
  Consider a sequence of varieties (see Notation~\ref{not:fields}):
  \begin{equation}\label{eq:x_var}
    X_0 := \V_{K/k_0}(\bar{c}) \supset X_1 := \V_{K/k_0\langle \bar{a}_1 \rangle}(\bar{c}) \supset X_2 := \V_{K/k_0\langle \bar{a}_1, \bar{a}_2 \rangle}(\bar{c}) \supset \ldots
  \end{equation}
  
  \paragraph{Claim:} \emph{For every $i \geqslant 0$, if $X_{i}$ and $X_{i + 1}$ have a common component $C$, then $C = \{\bar{c}\}$}.
  
 Let $i$ be such that $X_i$ and $X_{i+1}$ have a common component $C$.  For every $j \geqslant 1$, we introduce $Y_j := \V_{K/k_0\langle \bar{a}_j \rangle}(\bar{c})$.
  Since $X_{i + 1} \subset Y_{i + 1}$, we have $C \subset Y_{i + 1}$.
  We claim that, for every $j \geqslant i + 1$, $C \subset Y_j$.
  Lemma~\ref{lem:map_independent} implies that there exists a $k$-endomorphism $\alpha$ of $K$ such that $\alpha$ leaves $\bar{a}_1, \ldots, \bar{a}_i$ invariant and maps $\bar{a}_{i + 1}$ to $\bar{a}_j$.
  Since $C$ is a component of $X_i$, it is defined over $k_0\langle \bar{a}_1, \ldots, \bar{a}_i \rangle^{\alg}$, and therefore $C$  is $\alpha$-invariant.
  Thus, by applying $\alpha$ to the inclusion $C \subset Y_{i + 1}$, we obtain $C \subset Y_{j}$.
  
  By applying Lemma~\ref{lem:almost_intersection} iteratively to the component $C$ and $\bar{a} = (\bar{a}_1, \ldots, \bar{a}_{j})$ and $\bar{b} = \bar{a}_{j + 1}$ for $j = i + 1, i + 2, \ldots$, we show that $C$ is a component of $X_j$ for every $j \geqslant i + 1$.
  On the other hand, Proposition~\ref{prop:existence_N} implies that there exists $N$ such that $X_N = \{\bar{c}\}$.
  Thus, $C = \{\bar{c}\}$, and the claim is proved.
  
  Since $\bar{a}_1, \bar{a}_2, \ldots$ have the same ideals of definition over $k$ and 
  ideals of the form $\I_k(\bar{a}_1, \ldots, \bar{a}_s)$ are generated by the ideals of $\bar{a}_i$'s,
  we have
  \[
    F := \FD_k(\bar{a}_1) = \FD_k(\bar{a}_1, \bar{a}_2) = \FD_k(\bar{a}_1, \bar{a}_2, \bar{a}_3) = \ldots
  \]
  Therefore, for every $s \geqslant 0$,
  \begin{equation}\label{eq:trdeg_decomposition}
    \trdeg_{k_0\langle \bar{a}_1, \ldots, \bar{a}_s \rangle} k = \trdeg_{k_0\langle \bar{a}_1, \ldots, \bar{a}_s \rangle} \bar{c} + \trdeg_{k_0\langle \bar{c}, \bar{a}_1, \ldots, \bar{a}_s \rangle} k = \trdeg_{k_0\langle \bar{a}_1, \ldots, \bar{a}_s \rangle} \bar{c} + \trdeg_F k,
  \end{equation}
  where the latter equality is due to Lemma~\ref{lem:trdeg}.
  Thus, $r$ is the smallest integer such that 
  \[
    \trdeg_{k_0\langle \bar{a}_1, \ldots, \bar{a}_r \rangle} \bar{c} = \trdeg_{k_0\langle \bar{a}_1, \ldots, \bar{a}_{r + 1} \rangle} \bar{c}.
  \]
  Then $\adim X_r = \adim X_{r + 1}$, so every component of $X_{r + 1}$ is $\{\bar{c}\}$.
  Hence, $X_{r + 1} = \{\bar{c}\}$, and  so $\adim X_{r} = 0$.
  The fact that $\adim X_r = 0$ implies the first part of the proposition, and $X_{r + 1} = \{\bar{c}\}$ implies the second part of the proposition.
\end{proof}

\begin{proof}[Proof of Theorem~\ref{thm:num_exp}]
  Consider a generic solution \[(\bar{x}_1^\ast, \ldots, \bar{x}_{\ell + 1}^\ast, \bar{y}_1^\ast, \ldots, \bar{y}_{\ell + 1}^\ast, \bar{u}_1^\ast, \ldots, \bar{u}_{\ell + 1}^\ast)\] of $\Sigma_{\ell + 1}$.
  We apply Proposition~\ref{prop:bound_N} with $\bar{a}_i = (\bar{y}_i^\ast, \bar{u}_i^\ast)$ for every $1 \leqslant i \leqslant \ell + 1$, $k_0 = \C$, $k = \C(\bar{\mu})$.
  Since the sequence $\trdeg_{k_0\langle \bar{a}_1, \ldots, \bar{a}_i\rangle} k$ for $i = 0, \ldots, \ell + 1$ is nonincreasing, there will be $r \leqslant \ell$ as in Proposition~\ref{prop:bound_N}.
  Furthermore, it will be the same as $r$ in the statement of Theorem~\ref{thm:num_exp}.
  We have: \begin{itemize}
  \item $\FD_k(\bar{a}_1)$ is the field of globally ME-identifiable functions (by~\cite[Theorem~19]{allident} or Proposition~\ref{prop:existence_N}) and \item the field of locally ME-identifiable functions is algebraic over $\FD_k(\bar{a}_1)$.
  \end{itemize}
  Hence,  $r$ being the smallest number such that $\FD_k(\bar{a}_1)$ is algebraic over $k_0 \langle \bar{a}_1, \ldots, \bar{a}_r\rangle$ implies that $r$ is the smallest number such that
  the field of locally SE-identifiable functions of $\Sigma_r$ coincides with the field of locally ME-identifiable function in $\Sigma$.
  Thus, $\NEL(\Sigma) = r$.
  Finally, $\FD_k(\bar{a}_1) \subset k_0\langle\bar{a}_1, \ldots, \bar{a}_{r + 1} \rangle$ implies that ME-identifiable functions in $\Sigma$ are SE-identifiable in $\Sigma_{r + 1}$, so \[\NEG(\Sigma) \leqslant r + 1 = \NEL(\Sigma) + 1.\] 
\end{proof}



\section{Algorithm, implementation, and examples}\label{sec:alg}

\subsection{Algorithm: theory}\label{sec:alg_theory}

Theorem~\ref{thm:num_exp} implies the correctness of the following algorithm.

\begin{algorithm}[H]
\caption{Computing $\NEL(\Sigma)$ and estimating $\NEG(\Sigma)$}\label{alg:num_exp}
\begin{description}[labelwidth=15pt,leftmargin=\dimexpr\labelwidth+\labelsep\relax]
  \item[In:] 
  \begin{itemize}[leftmargin=*]
      \item an algebraic differential model $\Sigma$;
      \item (optional; for probabilistic version) real number $0 \leqslant p < 1$;
  \end{itemize}
 \item[Out:] positive integer $r$ such that $\NEL(\Sigma) = r$ and $\NEG(\Sigma) \in \{r, r + 1\}$.
 In the probabilistic version, this result will be correct with probability at least $p$.
\end{description}

Set $d_0 = \ell$. For $i = 1, 2, \ldots, \ell + 1$, do:

\begin{enumerate}[leftmargin=!,labelwidth=1.5em, label=\arabic*]

    \item Using Algorithm~\ref{alg:defect}, compute $d_i = \df(\Sigma_i)$ (see Definition~\ref{def:defect});\\
    \emph{(in the probabilistic version, the input probability for Algorithm~\ref{alg:defect} is $1 - \frac{1 - p}{\ell}$)}
    
    \item If $d_i = d_{i - 1}$, stop and return $i - 1$.
    
\end{enumerate}
\end{algorithm}

\begin{lemma}
  Algorithm~\ref{alg:num_exp} is correct.
\end{lemma}

\begin{proof}
  If the outputs of Algorithm~\ref{alg:defect} are correct, the returned result will be correct due to Theorem~\ref{thm:num_exp}.
  In the probabilistic version, the probability that at least one of the instances of Algorihm~\ref{alg:defect} will return wrong result does not exceed
   $ \ell \left(1 - \left(1 - \tfrac{1 - p}{\ell}\right)\right) = 1 - p$.
\end{proof}

Our algorithm for computing the identifiability defect will use, as a subroutine, algorithm(s) described in Theorem~\ref{thm:sed} below.

\begin{notation}
  We  call the \emph{complexity} of a
  model $\Sigma$ the maximum of the total number of variables (parameters, states, inputs, and outputs) and the length of a straight-line program (see~\cite[Chapter~4.1]{algebraic_complexity}) computing the numerators and denominators of the right-hand side of $\Sigma$.
 For measuring the complexity of the algorithms in this section, we use the notion of arithmetic complexity, that is the number of arithmetic operations in the ground field, see~\cite[Chapter~12]{Wigderson2019} for more details.
\end{notation}

\begin{theorem}[{\cite{Sedoglavic}}]\label{thm:sed}
Consider the following problem:
\begin{description}[leftmargin=2.5em]
  \item[In:] an algebraic differential model $\Sigma$ without parameters (that is, $\ell = 0$);
 \item[Out:] $\trdeg_{\C\langle \bar{y}^\ast, \bar{u}^\ast \rangle}\C(\bar{x}^\ast)$, where $(\bar{x}^\ast, \bar{y}^\ast, \bar{u}^\ast)$ is any generic solution of $\Sigma$.
\end{description}
Then
\begin{enumerate}
    \item\label{part:1} There exists a deterministic algorithm for solving this problem; 
    \item\label{part:2} There exists a probabilistic Monte Carlo algorithm with polynomial arithmetic complexity with respect to the complexity of $\Sigma$.
\end{enumerate}
\end{theorem}

\begin{proof}
  The theorem follows from the results from~\cite{Sedoglavic} as follows.
  Part~\ref{part:1} follows from~\cite[Corollary~2.1]{Sedoglavic}, in which $X$ is $\bar x^\ast$, $Y$ is $\bar{y}^\ast$ and $\mathcal G$ is $\C\langle \bar{y}^\ast, \bar{u}^\ast \rangle$.
  
  Part~\ref{part:2} follows from \cite[Theorem~1.1]{Sedoglavic} (together with a more precise complexity bound) as follows.
  The algorithm 
 whose existence is stated in~\cite[Theorem~1.1]{Sedoglavic} computes the 
  smallest
  number of nonobservable 
 state variables
 that are assumed to be known in order to make the system observable.
  The definition of observability~\cite[Section~2.1]{Sedoglavic} implies that  the components of $\bar x ^\ast$ corresponding to such  a set of state variables is a transcendence basis of $\C(\bar{x}^\ast)$ over $\C(\bar{y}^\ast, \bar{u}^\ast)$, so the cardinality of this set is the desired transcendence degree.
\end{proof}

\begin{algorithm}[H]
\caption{Computing $\df(\Sigma)$}\label{alg:defect}
\begin{description}[labelwidth=15pt,leftmargin=\dimexpr\labelwidth+\labelsep\relax]
  \item[In:] 
  \begin{itemize}[leftmargin=*]
      \item an algebraic differential model $\Sigma$;
      \item (optional; for probabilistic version) real number $0 \leqslant p < 1$;
  \end{itemize}
 \item[Out:] $\df(\Sigma)$. In the probabilistic version, this result is correct with probability at least $p$.
\end{description}

\begin{enumerate}[leftmargin=!,labelwidth=1.5em, label=\arabic*]
    \item Construct two parameter-free algebraic differential models:
    \begin{enumerate}
        \item $\Sigma'$ obtained from $\Sigma$ by viewing all parameters as state variables satisfying equations $\mu_i' = 0$ for every $\mu_i \in \bar{\mu}$;
        \item $\Sigma''$ obtained from $\Sigma'$ by adding a new output for each state variable corresponding to a parameter of $\Sigma$.
    \end{enumerate}
    
    \item Run any of the algorithms from Theorem~\ref{thm:sed} on $\Sigma'$ and $\Sigma''$, denote the results by $A$ and $B$, respectively
    \emph{(in the probabilistic version, the input probability is $\frac{1 + p}{2}$).}

    \item return $A - B$.
\end{enumerate}
\end{algorithm}

\begin{lemma}
  Algorithm~\ref{alg:defect} is correct.
\end{lemma}

\begin{proof}
  Let $(\bar{x}^\ast, \bar{y}^\ast, \bar{u}^\ast)$ be a generic solution of~\eqref{eq:sigma}.
  Then
  \[
    A = \trdeg_{\C\langle \bar{y}^\ast, \bar{u}^\ast \rangle} \C(\bar{x}^\ast, \bar{\mu}) \quad\text{ and }\quad B = \trdeg_{\C(\bar{\mu})\langle \bar{y}^\ast, \bar{u}^\ast \rangle} \C(\bar{x}^\ast, \bar{\mu}).
  \]
  Since one can compose a transcendence basis of $(\bar{x}^\ast, \bar{\mu})$ over $\C\langle \bar{y}^\ast, \bar{u}^\ast \rangle$ by first taking a transcendence basis of $\bar{\mu}$ over $\C\langle \bar{y}^\ast, \bar{u}^\ast \rangle$ which is of cardinality $\df(\Sigma)$ and then taking a transcendence basis of $\bar{x}^\ast$ over $\C(\bar{\mu})\langle \bar{y}^\ast, \bar{u}^\ast \rangle$, we have
  \[
  \df(\Sigma) = \trdeg_{\C\langle \bar{y}^\ast, \bar{u}^\ast \rangle} \C(\bar{\mu}) = \trdeg_{\C\langle \bar{y}^\ast, \bar{u}^\ast \rangle} \C(\bar{x}^\ast, \bar{\mu}) - \trdeg_{\C(\bar{\mu})\langle \bar{y}^\ast, \bar{u}^\ast \rangle} \C(\bar{x}^\ast, \bar{\mu}) = A - B.
  \]
  Hence, if both $A$ and $B$ have been computed correctly, the returned result is correct. 
  In the probabilistic version, the probability of at least one of them being incorrect does not exceed
  \[
  2 \left(1 - \tfrac{1 + p}{2}\right) = 1 - p.\qedhere
  \]
\end{proof}

\begin{proposition}\label{prop:complexity}
  If Algorithm~\ref{alg:defect} uses the second algorithm from Theorem~\ref{thm:sed}, then Algorithm~\ref{alg:num_exp} is a probabilistic Monte Carlo algorithm of polynomial arithmetic complexity with respect to the complexity of $\Sigma$.
\end{proposition}

\begin{proof}
  First we will prove that the arithmetic complexity of Algorithm~\ref{alg:defect} is polynomial.
  The first and the last steps have polynomial complexity.
  The fact that the arithmetic complexity of the second step is polynomial follows from Theorem~\ref{thm:sed} and the fact that the complexities of $\Sigma'$ and $\Sigma''$ are polynomial in the complexity of $\Sigma$.
  
  Let $\ell$ be the number of parameters.
  Since $d_i \leqslant \ell$ for every $i \geqslant 0$ except for the last and $d_0 > d_1 > d_2 > \ldots$, the counter $i$ in Algorithm~\ref{alg:num_exp} will not exceed $\ell + 1$.
  Then there will be at most $\ell + 1$ runs of Algorithm~\ref{alg:defect}, and each run will be on a system of complexity at most $\ell + 1$ times the complexity of $\Sigma$.
  Therefore, the total arithmetic complexity will be still polynomial in the complexity of $\Sigma$.
\end{proof}


\subsection{Algorithm: implementation and examples}\label{sec:impl_ex}

We  implemented the probabilistic version of Algorithm~\ref{alg:num_exp} for computing the bound from Theorem~\ref{thm:num_exp} in Julia language using Oscar and Nemo libraries~\cite{Nemo} together with a version of the algorithm by Sedoglavic from Theorem~\ref{thm:sed}.
The code and examples described below are available at~\url{https://github.com/pogudingleb/ExperimentsBound}.

Below we will demonstrate the algorithm and the bound on several examples and compare with the algorithm presented in~\cite{allident} (see Example~\ref{ex:comparison} and Table~\ref{tab:comparison}).
All of the runtimes reported below have been measured on a laptop with 1.6 GHz processor (Intel Core i5) and 16GB RAM.
All of the computations reported below have been performed with the correctness probability of $99\%$ (see the specification of Algorithm~\ref{alg:num_exp}).

\begin{remark}\label{rem:sian}
For some of the examples below, we were able to obtain the exact values of $\NEG(\Sigma)$ using SIAN~\cite{sian}.
SIAN is software that can check single experiment identifiability of any fixed function of parameters.
We used it as follows:
\begin{enumerate}
  \item If, for some $r$, all  parameters of $\Sigma_r$ are globally identifiable, then $\NEG(\Sigma) \leqslant r$.
  \item If, for some $r$, the parameter identifiability of $\Sigma_r$ and $\Sigma_{r + 1}$ are not the same, then $\NEG(\Sigma) \geqslant r + 1$.
\end{enumerate}
\end{remark}

\begin{example}[The counterexample from Section~\ref{sec:counterexample}]\label{ex:counter}
  In Section~\ref{sec:counterexample}, we have shown that $\NEG(\Sigma) > 1$ for the following system $\Sigma$:
  \[
  \begin{cases}
           x_1' = 0,\\
           x_2' = x_1x_2 + \mu_1 x_1 + \mu_2,\\
           y = x_2.
    \end{cases}
  \]
  Our implementation shows that $\NEL(\Sigma) = 2$ and $\NEG(\Sigma) \in \{2, 3\}$. 
  The computation took $0.01$~seconds.
  Using SIAN as describe in Remark~\ref{rem:sian}, we find that both parameters $\mu_1$ and $\mu_2$ are globally identifiable in $\Sigma_2$.
  Combining it with $\NEG(\Sigma) \in \{2, 3\}$ obtained by the algorithm, we conclude that \[\NEL(\Sigma) = \NEG(\Sigma) = 2.\]
  The same bound is given by~\cite[Theorem~21]{allident}. The computation took $0.3$~seconds.
\end{example}

\begin{example}[SEIR epidemiological model]\label{ex:SEIR}
  Consider the following SEIR model~\cite[Equation~(2.2)]{TL18}:
  \begin{equation}\label{eq:seir_original}
      \begin{cases}
          S' = -\beta \frac{SI}{N},\\
          E' = \beta\frac{SI}{N} - \nu E,\\
          I' = \nu E - \alpha I,\\
          R' = \alpha I,
      \end{cases}
  \end{equation}
  where $S$, $E$, $I$, $R$ are the numbers of individuals susceptible to the infection, exposed, infected, and recovered, respectively, and $N := S + E + I + R$ is the total population which is known.
  Note that~\eqref{eq:seir_original} implies that $N' = 0$.
  The output we will consider will be $\gamma I + \delta E$, where $\gamma$ and $\delta$ are constants corresponding to  factors such as, for instance, accuracy of the tests for the infection or the percentage of individuals going to a doctor after noticing the symptoms.
  We will assume that there are several experiments with the same values of $\alpha, \beta, \nu, \delta$ but varying values of $\gamma$ (e.g., before and after improving the accuracy of the test).
  
  To encode these assumptions into our framework, we will make $\gamma$  a constant state variable and add an output for it.
  We will also replace the equation for $R$ from~\eqref{eq:seir_original} with $N' = 0$ as $R$ does not appear in other equations other than inside $N$.
  This yields the following model $\Sigma$:
  \begin{equation}
      \begin{cases}
          S' = -\beta \frac{SI}{N},\\
          E' = \beta\frac{SI}{N} - \nu E,\\
          I' = \nu E - \alpha I,\\
          N' =
          \gamma' = 0,\\
          y_1 = \gamma I + \delta E,\\
          y_2 = \gamma,\; y_3 = N.
      \end{cases}
  \end{equation}
  Our implementation shows that $\NEL(\Sigma) = 1$ and $\NEG(\Sigma) \in \{1, 2\}$. 
  The computation took $0.05$ seconds.
  Using SIAN as described in Remark~\ref{rem:sian}, we find that all the parameters are only locally identifiable from a single experiment but become globally identifiable after
  2 experiments. 
  Therefore, $\NEG(\Sigma) = 2$, so the bound given by the algorithm is exact in this case.
  
  The program for computing a bound for the number of experiments provided in~\cite{allident} did not finish on this example after two hours of computation.
\end{example}

\begin{example}[Linear compartment models with controlled rates]\label{ex:lincomp}
Linear compartment models typically represent a set of compartments in which material is transferred  from some compartments to other compartments.  
It is also allowed to have a leakage of material from some compartments out of the system and input of material into some compartments from outside the system.

Linear compartment models are typically represented as directed graphs with edges labeled by scalar parameters (called rate constants).
An example of such a representation is shown in Figure~\ref{lincomp_ex}.
The rules of transforming such a graph into a system of ODEs are the following:
\begin{itemize}
    \item (compartments) each vertex of the graph correspond to a state variable (a compartment);
    \item (transfers) for each edge $i \to j$ with a rate constant $a_{ji}$, we add a term $a_{ji} x_i$ to the equation for $x_j'$ and a term $-a_{ji} x_i$ to the equation $x_i'$ (the corresponding terms for the edge $1 \to 2$ on Figure~\ref{lincomp_ex} are underlined in the system);
    \item (leaks) for each edge from vertex $i$ without a target (such as an edge from vertex~1 in Figure~\ref{lincomp_ex}) with a rate constant $a_{0i}$, we add a term $-a_{0i}x_i$ to the equation for $x_i'$ (such a term for $a_{01}$ is in boldface in Figure~\ref{lincomp_ex});
    \item (outputs) outgoing edge with a small circle at the end marks  state variables taken as outputs (e.g., $x_1$ in Figure~\ref{lincomp_ex});
    \item (inputs) for an incoming edge without a source (such as the one pointing at node 3 in Figure~\ref{lincomp_ex}), we add an input variable to the corresponding compartment (added variable $u$ in the equation for $x_3'$).
\end{itemize}

\setcounter{table}{0}
\begin{table}[H]
    \centering
    \begin{tabularx}{\linewidth}{Z|Z}
\begin{tikzpicture}
         \node (1) at (1, 1) [circle, draw=black] {1};
         \node (2) at (1, 3) [circle, draw=black] {2};
         \node (3) at (3, 1) [circle, draw=black] {3};
         \node (leak) at (0, 0) [draw=none] {};
         \node (out) at (1, 0) [circle, draw=black] {};
         \node (in) at (3, 2) [draw=none] {};
         \draw[decoration={markings,mark=at position 1 with
    {\arrow[scale=1,>=stealth]{>}}},postaction={decorate}] ([xshift=1mm]1.north) -- ([xshift=1mm]2.south) node[right, pos=0.5, font=\Large] {$a_{21}$};
         \draw[decoration={markings,mark=at position 1 with
    {\arrow[scale=1,>=stealth]{>}}},postaction={decorate}] ([xshift=-1mm]2.south) -- ([xshift=-1mm]1.north) node[left, pos=0.5, font=\Large] {$a_{12}$};
          \draw[decoration={markings,mark=at position 1 with
    {\arrow[scale=1,>=stealth]{>}}},postaction={decorate}] (3) -- (1) node[above, pos=.5, font=\Large] {$a_{13}$};
           \draw[decoration={markings,mark=at position 1 with
    {\arrow[scale=1,>=stealth]{>}}},postaction={decorate}] (1) -- (leak) node[above, sloped, pos=.5, font=\Large] {$a_{01}$};
           \draw[decoration={markings,mark=at position 1 with
    {\arrow[scale=1,>=stealth]{>}}},postaction={decorate}] (1) -- (out);
            \draw[decoration={markings,mark=at position 1 with
    {\arrow[scale=1,>=stealth]{>}}},postaction={decorate}] (in) -- (3);
      \end{tikzpicture}     
      &    
        {\begin{equation*}
        \begin{cases}
          x_1' = \mathbf{-a_{01}x_1} - \underline{a_{21}x_1} + a_{12}x_2 + a_{13}x_3\\
          x_2' = \underline{a_{21} x_1} - a_{12}x_2\\
          x_3' = - a_{13}x_3 + u\\
          y = x_1
        \end{cases}\end{equation*}}
\end{tabularx}
\captionsetup{name=Figure}
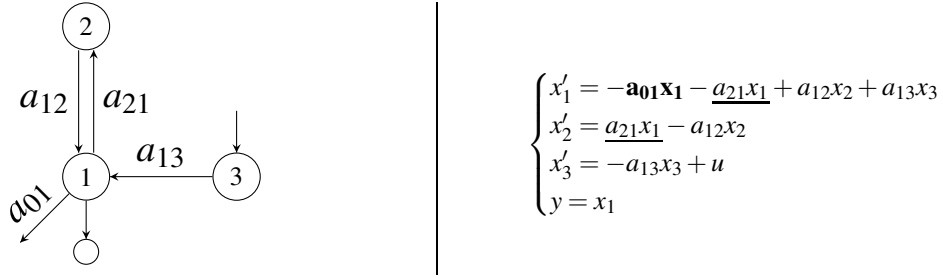
\caption{Example of a linear compartment ODE model and of the corresponding graph}\label{lincomp_ex}
\end{table}

\vspace{-0.2in}
We will consider three series of models: cyclic, catenary, and mammilary. These linear compartment models and their modifications have
 recently been actively studied from the identifiability perspective~~\cite{vdH98, GMS17, GHMS2019, GOS20}. 
The corresponding graphs are given in Figure~\ref{fig:lincomp_series}.
Since these models are linear and have a single output, \cite[Theorem~1]{OPT19} together with~\cite[Theorem~21]{allident} implies that $\NEG(\Sigma) = \NEL(\Sigma) = 1$ for every such model $\Sigma$.
\setcounter{figure}{1}
\begin{figure}[H]
    \centering
    
    \begin{minipage}{0.56\textwidth}
    \centering
      \subcaptionbox{Cycle model\label{fig:cycle}}[\linewidth][c]{\begin{tikzpicture}
	\begin{pgfonlayer}{nodelayer}
		\node [style=emptycircle] (0) at (-5, 0.5) {1};
		\node [style=emptycircle] (1) at (-2.25, 2.75) {2};
		\node [style=emptycircle] (2) at (1.75, 2.75) {3};
		\node [style=emptycircle] (3) at (-2.5, -1.75) {$n$};
		\node [style=emptycircle] (4) at (2, -1.75) {$n - 1$};
		\node [style=none] (5) at (3.5, 2) {};
		\node [style=none] (6) at (3.5, -0.5) {};
		\node [style=none] (7) at (3.5, 0.75) {$\vdots$};
		\node [style=none] (8) at (-0.25, 3.25) {$a_{32}$};
		\node [style=none] (9) at (-4.5, 2) {$a_{21}$};
		\node [style=none] (10) at (-4, -1) {$a_{1\,n}$};
		\node [style=none] (11) at (-0.25, -2.25) {$a_{n\, (n - 1)}$};
		\node [style=none] (12) at (-7.75, 0.5) {};
		\node [style=none] (13) at (-6.5, 1) {$a_{01}$};
		\node [style=output] (14) at (-6, -2) {};
	\end{pgfonlayer}
	\begin{pgfonlayer}{edgelayer}
		\draw [style=newstyle] (0) to (1);
		\draw [style=newstyle] (1) to (2);
		\draw [style=newstyle] (3) to (0);
		\draw [style=newstyle] (4) to (3);
		\draw [style=newstyle] (2) to (5.center);
		\draw [style=newstyle] (6.center) to (4);
		\draw [style=newstyle] (0) to (12.center);
		\draw [style=newstyle] (0) to (14);
	\end{pgfonlayer}
\end{tikzpicture}}

      \subcaptionbox{Catenary model}{\begin{tikzpicture}
	\begin{pgfonlayer}{nodelayer}
		\node [style=emptycircle] (0) at (-9.75, 2) {$1$};
		\node [style=emptycircle] (1) at (-5, 2) {$2$};
		\node [style=emptycircle] (2) at (0, 2) {$n - 1$};
		\node [style=emptycircle] (3) at (4.75, 2) {$n$};
		\node [style=none] (4) at (-3.5, 2.5) {};
		\node [style=none] (5) at (-3.5, 1.5) {};
		\node [style=none] (6) at (-1.5, 2.5) {};
		\node [style=none] (7) at (-1.5, 1.5) {};
		\node [style=none] (8) at (-7.25, 3) {$a_{21}$};
		\node [style=none] (9) at (-7.25, 1.25) {$a_{12}$};
		\node [style=none] (10) at (2.5, 1.25) {$a_{(n - 1)\,n}$};
		\node [style=none] (11) at (2.5, 3) {$a_{n\, (n - 1)}$};
		\node [style=none] (12) at (-2.5, 2) {$\cdots$};
		\node [style=output] (13) at (-11, 3) {};
		\node [style=output] (14) at (-11, 3) {};
	\end{pgfonlayer}
	\begin{pgfonlayer}{edgelayer}
		\draw [style=newstyle, bend left=15] (0) to (1);
		\draw [style=newstyle, bend left=15] (1) to (0);
		\draw [style=newstyle, bend left=15] (2) to (3);
		\draw [style=newstyle, bend left=15] (3) to (2);
		\draw [style=newstyle, bend left=15, looseness=1.25] (1) to (4.center);
		\draw [style=newstyle, bend right=330] (5.center) to (1);
		\draw [style=newstyle, bend left=15] (6.center) to (2);
		\draw [style=newstyle, bend left=15] (2) to (7.center);
		\draw [style=newstyle] (0) to (14);
	\end{pgfonlayer}
\end{tikzpicture}\vspace{-.5cm}}
    \end{minipage}
    \vspace{-.1in}
    \begin{minipage}{0.4\textwidth}
      \subcaptionbox{Mammilary model}{\begin{tikzpicture}
	\begin{pgfonlayer}{nodelayer}
		\node [style=emptycircle] (0) at (-5, 0) {$1$};
		\node [style=emptycircle] (1) at (1, 6) {$2$};
		\node [style=emptycircle] (2) at (1, 1) {$3$};
		\node [style=emptycircle] (3) at (1, -4) {$n$};
		\node [style=none] (4) at (-3.25, 4) {$a_{21}$};
		\node [style=none] (5) at (0, 3.25) {$a_{12}$};
		\node [style=none] (6) at (-1.5, 1.5) {$a_{31}$};
		\node [style=none] (7) at (-1.75, -0.5) {$a_{13}$};
		\node [style=none] (8) at (-0.75, -1.5) {$a_{n\,1}$};
		\node [style=none] (9) at (-3.25, -2.5) {$a_{1\,n}$};
		\node [style=none] (10) at (1, -1.25) {$\vdots$};
		\node [style=output] (11) at (-7, 0) {};
	\end{pgfonlayer}
	\begin{pgfonlayer}{edgelayer}
		\draw [style=newstyle, bend left=15] (0) to (1);
		\draw [style=newstyle, bend left=15] (0) to (2);
		\draw [style=newstyle, bend left=15] (0) to (3);
		\draw [style=newstyle, bend left=15] (1) to (0);
		\draw [style=newstyle, bend left=15] (2) to (0);
		\draw [style=newstyle, bend left=15] (3) to (0);
		\draw [style=newstyle] (0) to (11);
	\end{pgfonlayer}
\end{tikzpicture}}
    \end{minipage}
    
    \caption{ Considered classes of linear compartment models represented by their graphs}
    \label{fig:lincomp_series}
\end{figure}
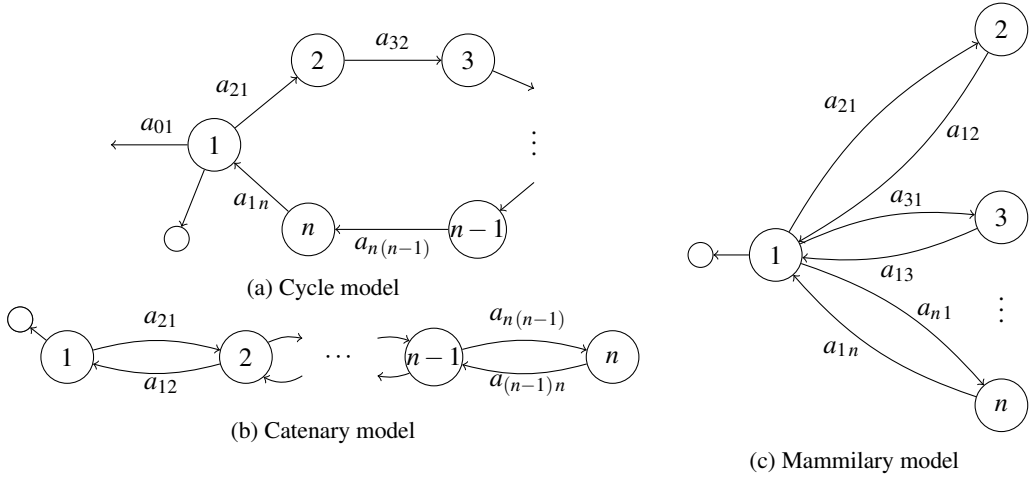
\vspace{-.178in}
We will consider a modification of these models similar to~\cite[Section~III.B]{VECB}.
The modification is motivated by voltage clamp protocols used to identify parameters in ion channel models~\cite{FN09}.
Ion channel models are often modelled using Markov models, which are similar to linear compartment models but with parameters depending on input functions (see \cite[Section~III.B]{VECB} and~\cite[Section~3.2]{RS06}).
In the context of ion channel models, it may be nonrealistic to include a generic time-dependent input into the model.
Instead of this, several experiments are performed such that the parameters depend on a constant input that takes different values for different experiments~\cite[\S 5 and \S 7]{FN09}.
Such a constant input can be encoded into our framework by adding a new state variable $x_0$ satisfying $x_0' = 0$ and a new output equal to $x_0$.
We will consider the case in which all of the parameters depend linearly on the constant input $x_0$, that is:
  $a_{ij} = b_{ij} + c_{ij} x_0$
for all $i$ and $j$, where $b_{ij}$ and $c_{ij}$ are 
new parameters.
A dependence of this form was used for some of the parameters in~\cite[Supplementary Material]{CR02} and can be viewed as a linear approximation to the dependencies used in~\cite{VECB, FN09}.
For example, the cycle model (Figure~\ref{fig:cycle}) with $n = 4$ will be represented as shown on Figure~\ref{cycle4} (cf.~\cite[III.B]{VECB})
\setcounter{table}{2}
\begin{table}[H]
\def\arraystretch{.3}
    \centering
    \begin{tabularx}{\linewidth}{Z|Z}
      \ctikzfig{cycle4}
      &
        {\begin{equation*}\begin{cases}
          x_0' = 0\\
          x_1' = (b_{14} + c_{14}x_0) x_4 - (b_{21} + c_{21}x_0) x_2\\
          x_2' = (b_{21} + c_{21}x_0) x_1 - (b_{32} + c_{32}x_0) x_3\\
          x_3' = (b_{32} + c_{32}x_0) x_2 - (b_{43} + c_{43}x_0) x_4\\
          x_4' = (b_{43} + c_{43}x_0) x_3 - (b_{14} + c_{14}x_0) x_1\\
          y_1 = x_0,\; y_2 = x_1
        \end{cases}\end{equation*}}
\end{tabularx}
\captionsetup{name=Figure}
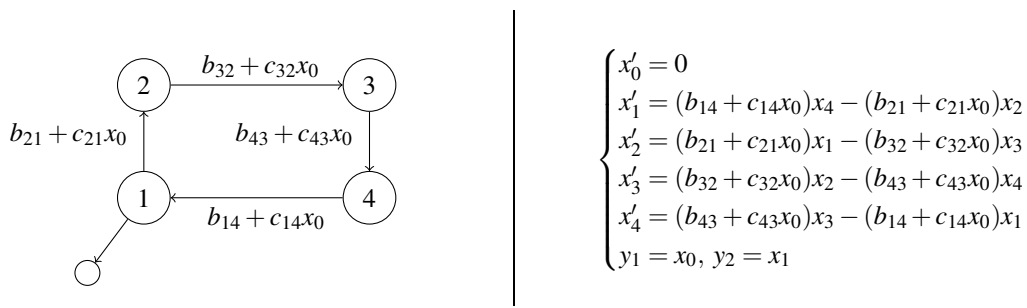
\caption{Cyclic model with $n = 4$ compartments with constant input in the reaction rates: graph (left) and ODE model (right)}\label{cycle4}
\end{table}
\vspace{-0.2in}
We have analysed models from families in Figure~\ref{fig:lincomp_series} with introduced constant input $x_0$ as described above using our implementation.
The resulting values of the bound and the runtimes are summarized in Table~\ref{table:lincomp_runtimes}.
The algorithm for computing a bound for the number of experiments described in~\cite[Remark~22]{allident} did not finish on any of the models even for $n = 3$ after two hours of computation.

\setcounter{table}{1}
\setlength{\tabcolsep}{4.5pt}
\renewcommand{\arraystretch}{1.75}
\begin{table}[H]
\centering
\begin{tabular}{|l|c|c|c|c|c|c||c|}
\hline
    \multirow{2}{*}{\textbf{Model}} & \multicolumn{2}{|c|}{$\NEL(\Sigma)$} &  \multicolumn{2}{|c|}{$\NEG(\Sigma) \in$} & \multicolumn{2}{|c|}{\textbf{runtime (sec.)}} & \textbf{max $n$ feasible} \\
     \cline{2-7}
     & $n = 3$ & $4 \leqslant n \leqslant 15$ & $n = 3$ & $4 \leqslant n \leqslant 15$ & $n = 10$ & $n = 15$ & \textbf{for SIAN${}^*$} \\
     \hline
     \hline
     Cycle & $3$ & $3$ & $\{3, 4\}$ & $\{3, 4\}$ & $9.5$ & $41$ & $4$ \\
     \hline
     Catenary & $4$ & $5$ & $\{4, 5\}$ & $\{5, 6\}$ & $45.6$ & $330$ & $3$\\
     \hline
     Mammilary & $4$ & $5$ & $\{4, 5\}$ & $\{5, 6\}$ & $45.8$ & $320$ & $3$\\
     \hline
 \end{tabular}    

    \caption{ Results and runtimes of our implementation on cyclic, catenary, and mammilary models (see Figure~\ref{fig:lincomp_series}) with a constant input acting on reaction rates as on Figure~\ref{cycle4}.\\
    ${}^\ast$: for details on SIAN usage in this case, see Remark~\ref{rem:sian_lincomp}}
    \label{table:lincomp_runtimes}
\end{table}

\begin{remark}\label{rem:sian_lincomp}
For $\NEG(\Sigma)$, we tried to refine the result to obtain the exact value using SIAN~\cite{sian} as described in Remark~\ref{rem:sian}. The results are the following:
\begin{itemize}
    \item for the cycle model, we have found that $\NEG(\Sigma) \geqslant 4$ for $n = 3, 4$ as described in the second item of Remark~\ref{rem:sian}. 
    Combined with the bound given by our implementation, we obtain $\NEG(\Sigma) = 4$ for $n = 3, 4$, so the bound is exact in this case.
    Already for $n = 5$, the computation with SIAN did not finish in 10 hours on a server and used more than 20GB of memory.
    \item for the catenary and mammilary models, a computation with SIAN showed that, for $n = 3$,
    none
    of the individual parameters was identifiable after 5 experiments.
    Therefore, we cannot use SIAN to determine the exact bound in the way we did it for the cycle model.
    For $n = 4$, the computation with SIAN did not finish in 10 hours on a server and used more than 40GB of memory.
\end{itemize}
\end{remark}

\begin{remark}For all three series of models, the output of the algorithm stabilizes from $n = 4$.
It is natural to conjecture that the result will be the same for all larger values of $n$.
It would be interesting to have a mathematical argument showing this or maybe even a formula for the number of experiments in terms of numerical characteristics of the graph of a model.
\end{remark}

\end{example}

\begin{remark} 
The same procedure of linearly perturbing the rate constants can be applied to general chemical reaction networks, which yield, in general, highly nonlinear ODEs. In this setup, we also observe that the necessary number of experiments may become larger than $1$: for example, for the perturbed version of the phosphorylation model \cite[Example~6.1]{HOPY2020}, we get $\NEL(\Sigma) =2$ and $\NEG(\Sigma)\in \{2,3\}$.
\end{remark}


\begin{example}[Examples from~\cite{allident}]\label{ex:comparison}
  As we  mentioned, the algorithm from~\cite[Remark~22]{allident} does not produce any bound for Examples~\ref{ex:SEIR} and~\ref{ex:lincomp} in  reasonable time.
  For the sake of comparison, we run our algorithm on the examples collected in~\cite[Section~5]{allident}.
  The comparison is reported in Table~\ref{tab:comparison}, in which we also included the above examples for completeness.
  
  \begin{table}[H]
  \setlength{\tabcolsep}{2.2pt}
\renewcommand{\arraystretch}{1.3}
  \centering
  \begin{tabular}{|l|c|c|c|c|}
  \hline
       \multirow{2}{*}{\textbf{Model}} & \multicolumn{2}{|c|}{\textbf{Approach from~\cite{allident}}} & \multicolumn{2}{|c|}{\textbf{Our algorithm}}  \\
    \cline{2-5}
        & \textbf{time (sec.)} & \textbf{bound} & \textbf{time (sec.)} & \textbf{bound} \\
    \hline
    \hline
       Lotka-Volterra w/control~\cite[Section~5.1]{allident} & $0.3$ & $2$ & $0.005$ & $2$ \\
    \hline
         Slow-fast ambiguity~\cite[Section~5.2]{allident} & $0.37$ & $2$ & $0.024$ & $2$ \\
    \hline
         Lotka-Volterra w/mixture~\cite[Section~5.3]{allident} & $15$ & $4^\ast$ & $0.01$ & $2$ \\
    \hline
         SEIR - prevalence~\cite[Section~5.4]{allident} & $1$ & $1$ & $0.021$ & $2$ \\
    \hline
         SEIR - incidence~\cite[Section~5.4]{allident} & $340$ & $1$ & $0.032$ & $2$ \\
    \hline
    \hline
         Counterexample from Section~\ref{sec:counterexample} (Example~\ref{ex:counter}) & $0.3$ & $2$ & $0.01$ & $3$ \\
    \hline
         SEIR w/ mixture (Example~\ref{ex:SEIR}) & $> 2\text{ h.}$ & N/A & $0.05$ & $2$ \\
    \hline
         Cycle for $n = 3^{\ast\ast}$ (Example~\ref{ex:lincomp}) & $> 2\text{ h.}$ & N/A & $0.3$ & $4$ \\
    \hline
         Catenary for $n = 3^{\ast\ast}$ (Example~\ref{ex:lincomp}) & $> 2\text{ h.}$ & N/A & $0.6$ & $5$ \\
    \hline
        Mammilary for $n = 3^{\ast\ast}$ (Example~\ref{ex:lincomp}) & $> 2\text{ h.}$ & N/A & $0.6$ & $5$ \\
    \hline
  \end{tabular}
      \caption{Comparison of bounds for $\NEG$ and runtimes with~\cite{allident}.\\
      ${}^{\ast}$:  obtained by a modification of the method, see~\cite[Section~5.3]{allident}; \cite[Theorem~21]{allident} gives $35$\\
      ${}^{\ast\ast}$: our algorithm from the present paper can tackle larger $n$ as well, see Table~\ref{table:lincomp_runtimes}}
      \label{tab:comparison}
  \end{table}
\end{example}

\section{Model theory and identifiability}\label{sec:model_theory}

The goal of this section is to explain the connections between identifiability and model theory and give an idea how the algebraic arguments from the preceding sections have been inspired and informed by model-theoretic considerations.
The section is structured as follows.
Section~\ref{sec:mt_setup} introduces some fundamental notions of model theory in the context of differential fields and explains their close relations with the concept of identifiability; the section culminates in the identifiability--model theory dictionary in Table~\ref{tab:dictionary}.
Section~\ref{sec:cemodeltheory} is about some model theory of differential fields behind \eqref{eq:counterexsys} and Conclusion~\ref{thm:example}. Specifically we describe canonical bases of types over constant fields.
In Section~\ref{sec:mtpf}, we give model-theoretic proofs of the key ingredients of the proof of Theorem~\ref{thm:num_exp}, Propositions~\ref{prop:existence_N} and~\ref{prop:bound_N}.
Some of the ideas from these proofs were crucial in the algebraic proof of Theorem~\ref{thm:num_exp}.

 Model theory is a kind of abstract algebra, which gives a common environment and common tools for studying algebraic structures such as group, fields, and fields equipped with derivations or automorphisms. Among these tools are {\em canonical bases} coming from an area of model theory called stability theory, and which implicitly play an important role in this paper.

\subsection{Setup}\label{sec:mt_setup}
We will use basic notions from model theory (such as language, structure, theory, and model). The reader is referred to~\cite[Chapter~1]{Marker_book} for additional details.
In this section, we will introduce relevant notions from model theory. We will specialize some of them for simplicity to the case of differential fields  and explain their relation to the identifiability problem.
Further details can be found in~\cite{MarkerMTDF} (also~\cite{Marker_book, Pillay_book}).
The correspondence between  notions from identifiability and model theory is summarized in Table~\ref{tab:dictionary}.
\begin{table}[H]
\setlength{\tabcolsep}{5pt}
    \centering
    \def\arraystretch{1.2}
    \begin{tabular}{|l|l|}
    \hline
         \textbf{Identifiability} & \textbf{Model theory} \\
         \hline
         \hline
         \begin{tabular}{l}Solution of~\eqref{eq:sigma}\end{tabular} & \begin{tabular}{l}Realization of the  formula~\eqref{eq:sigma} \end{tabular} \\
         \hline
         \begin{tabular}{l}Generic solution $(\bar{x}^\ast, \bar{y}^\ast, \bar{u}^\ast)$\\ of~\eqref{eq:sigma}\end{tabular} & \begin{tabular}{l}Realization of the generic type defined by~\eqref{eq:sigma}\\ (Example~\ref{ex:type_generic})\end{tabular} \\
         \hline
         \begin{tabular}{l}Identifiability (Definition~\ref{def:ident_si})\end{tabular} & \begin{tabular}{l}Definability over $(\bar{y}^\ast, \bar{u}^\ast)$ (Example~\ref{ex:ident_si})\end{tabular}\\
         \hline
         \begin{tabular}{l}Multi-experiment identifiability\\ (Definition~\ref{def:ident:multi})\end{tabular} & \begin{tabular}{l}
              Definability over some finite number of independent\\ realizations of $\tp((\bar{y}^\ast, \bar{u}^\ast) / \C(\bar{\mu}))$ (Example~\ref{ex:ident_mult})
         \end{tabular}\\
         \hline
         \begin{tabular}{l}Field of multi-experiment\\ identifiable functions\end{tabular} &
         
         \begin{tabular}{l}
              The canonical base $\cb((\bar{y}^\ast, \bar{u}^\ast) / \C(\bar{\mu}))$\\ (Example~\ref{ex:ident_me} and Propositions~\ref{prop:existence_N} and~\ref{prop:existence_N_mt})
         \end{tabular}\\
         \hline
    \end{tabular}
    \caption{ Identifiability - Model theory dictionary}
    \label{tab:dictionary}
\end{table}

One expresses system~\eqref{eq:sigma} by a formula in the appropriate language (i.e. the conjunction of the system of finitely many equations) as defined below.

\begin{our_definition}[Extensions of languages]
  For a language $\mathcal{L}$, a structure $M$ in $\mathcal{L}$, and a subset $A \subseteq M$, let $\mathcal{L}_A$ denote the extension of language $\mathcal{L}$ by adding a constant for each element of~$A$. 
\end{our_definition}
The reader should be careful to distinguish ``constants" in the sense of constant symbols in logic from constants in the sense of elements of a differential field on which the derivation is zero. 

\begin{example}[Language of differential fields]
  We  work in the language of differential fields 
   \[
   \mathcal{L}_{DF} := \{+, \cdot, ', 0, 1\}.
   \]
   This language allows one to express differential equations with rational coefficients (not arbitrary complex numbers as in~\eqref{eq:sigma}), and this is not sufficient to write a system of the form~\eqref{eq:sigma}.
   However, every equation in~\eqref{eq:sigma} is a formula in $\mathcal{L}_{DF, \mathbb{C}}$ in variables $\bar{x}, \bar{y}, \bar{u}, \bar{\mu}$.
\end{example}

We will work not in the theory of differential fields but in the theory of differentially closed fields. This ensures that the equations of interest have sufficiently many solutions.

\begin{our_definition}[{Differentially closed fields, \cite[Definition~4.3.29]{Marker_book}}]
  A differential field $K$ is called \emph{differentially closed} if, for all differential polynomials $f, g \in K\{x\} \setminus \{0\}$ with $\ord f > \ord g$, there is $a \in K$ such that $f(a) = 0$ and $g(a) \neq 0$ (for $p \in K\setminus \{0\}$, we define $\ord p := -1$).
  
  These fields share many properties with algebraically closed fields such as the Nullstellensatz: if a system of equations over $K$ has a solution in some extension of $K$, then it has a solution in $K$ as well~\cite[Corollary~2.6]{MarkerMTDF}.
  The property of being differentially closed can be written as a list of axioms in $\mathcal{L}_{DF}$. We denote the resulting theory of differentially closed fields by $\DCF_0$.
\end{our_definition}

Once we have put the system~\eqref{eq:sigma} into the model-theoretic context, we would like to be able to talk about its solutions and generic solutions.
This is done using the language of formulas and types.

\begin{our_definition}[{Types, \cite[Definition~4.1.1]{Marker_book}}]\label{def:types}
Let $M$ be an $\mathcal{L}$-structure and $A$ a subset of 
$M$. 
Then an $n$-{\em type} over $A$, relative to the structure $M$, is a set $\Phi$ of formulas in $\mathcal{L}_A$ with free variables $x_1,\ldots,x_n$ such that there exists an $\mathcal{L}_A$-structure $N$ containing $M$ (could be equal to $M$) such that
\begin{itemize}
    \item all $\mathcal{L}_M$-sentences true in $M$ are also true in $N$ (such $N$ is called \emph{an elementary extension} and we write $M\prec N$);
    \item there exist $a_1, \ldots, a_n \in N$ satisfying all the formulas in $\Phi$.
\end{itemize}
   Such a set or tuple $a_1, \ldots, a_n$ is called \emph{a realization of the type}. 
\end{our_definition}

\begin{example}[Important classes of types]\label{ex:types}
In this paper, we will encounter mostly types of the following forms:
\begin{itemize}
    \item Types defined by finitely many formulas (that is, $|\Phi| < \infty$ in Definition~\ref{def:types}).
    For example,  the system~\eqref{eq:sigma} (or any other system of differential-algebraic equations) defines such a type over any differential field containing at least one solution of~\eqref{eq:sigma}.
    Using conjunction, every such type can be defined by a single formula.
    \item Let $M$ be a model, $A \subset M$ be any subset, and $\bar{a}$ be a tuple from $M$.
    Then $\tp_M(\bar{a} / A)$ denotes the set of all formulas in $\mathcal{L}_A$ satisfied by $\bar{a}$ in $M$.
    Note that if $M\prec N$, then $\tp_N(\bar{a} / A) = \tp_M(\bar{a} / A)$. 
    \item
    Let $M$ be an $\mathcal{L}$-structure, $A$ a subset of $M$, and $p$ an $n$-type over $A$  relative to the structure $M$. Then $p$ is complete if, for every $\phi(\bar x)$ in $\mathcal{L}_A$, either $\phi$ or $\neg\phi$ is in $p$.
   The complete $n$-types over $A$ relative to $M$ are precisely of the form $\tp_N(\bar{a} / A)$ for $N$ an elementary extension of $M$.
    
For any automorphism $f$ of $M$, we define a map on the set of complete types over $M$ by 
    applying $f$ to
    the formulas contained in the types.
\end{itemize}
\end{example}

\begin{remark}[Types in differentially closed field]\label{rem:qe}
The theory $\DCF_0$ admits quantifier elimination~\cite[Theorem~4.3.32]{Marker_book}, that is, for every formula $\phi$, there is a quantifier-free formula equivalent to  
  $\phi$ 
  in $\DCF_0$.
   Therefore, every type can be defined by a set of quantifier-free formulas.
   
   In particular, if $K$ and $L$ are differentially closed fields and $A, \bar{a} \subset K\cap L$, then the types $\tp_K(\bar{a} / A)$ and $\tp_L(\bar{a} / A)$ are the same.
   Hence, working in the context of differentially closed fields, we will write simply $\tp(\bar{a} / A)$ without specifying the ambient differentially closed field.
\end{remark}

\begin{example}\label{ex:type_generic}
  Consider a generic solution $(\bar{x}^\ast, \bar{y}^\ast, \bar{u}^\ast)$ of~\eqref{eq:sigma} (see Definition~\ref{def:generic_solution}) in a differentially closed field $K \supset \C(\bar{\mu})$.
  Then we will call
  $\tp((\bar{x}^\ast, \bar{y}^\ast, \bar{u}^\ast) / \C(\bar{\mu}))$ \emph{the type of a generic solution of~\eqref{eq:sigma}}.
  This type contains all equations~\eqref{eq:sigma}, but also, for example, 
 any inequation (say, $x_1' \neq 0$) that is true for at least one solution of~\eqref{eq:sigma} and thus must be true for a generic one.
\end{example}

Model theory provides tools to construct large enough differential fields containing realizations of types of generic solutions of all the systems of interest (and in many copies so that we can talk about multiple experiments as well).

\begin{our_definition}[{Saturation, \cite[Definition~4.3.1]{Marker_book}}]\label{def:saturated}
  Let $\kappa$ be an infinite cardinal.
  A model $M$ of theory $T$ is called \emph{$\kappa$-saturated} if every complete type $\Phi$ such that 
  \[
  |\{ m\in M \mid m \text{ appears in } \Phi \}| < \kappa
  \]
  has a realization in $M$.
  $M$ is called \emph{saturated} if it is $|M|$-saturated.
\end{our_definition}

\begin{remark}\label{rem:types_mapping}
  If $M$ is a model of theory $T$ and $A \subset M$ is a subset and
  $M$ is saturated with $|M| > \max(|A|, |T|)$, then, for every $\bar{a}, \bar{b}$ in $M$, 
  \[
    \tp_M(\bar{a} / A) = \tp_M(\bar{b} / A) \iff \exists\; \text{automorphism }\alpha \colon M \to M \text{ such that } \alpha(\bar{a}) = \bar{b} \text{ and } \alpha|_{A} = \operatorname{id}
    \] 
    (see~\cite[Propositions~4.2.13 and~4.3.3]{Marker_book}).
\end{remark}

Now we define identifiability in the language of model theory.

\begin{our_definition}[{Definability, \cite[Definition 1.3.1]{Marker_book}}]
  A subset $X \subset M^n$ of a structure $M$ in a language $\mathcal{L}$ is called \emph{definable over} a subset $A \subset M$ if there exists a first-order formula $\phi(x_1, \ldots, x_n)$ in $\mathcal{L}_A$ such that
  \[
  (a_1, \ldots, a_n) \in X \iff \phi(a_1, \ldots, a_n) \text{ is true in $M$.}
  \]
\end{our_definition}

\begin{example}\label{ex:ident_si}
  Let $K$ be a differential field over a differential subfield $k_0$, and $\bar{a}$ and $\bar{b}$ are tuples of elements of $K$.
 It follows from \cite[Proposition~1.3.5]{Marker_book} and \cite[Theorem~2.6]{Kaplansky}
  that 
  \[
  \bar{a} \text{ definable over } \bar{b} \text{ in } \mathcal{L}_{DF, k_0} \iff \bar{a} \in k_0\langle \bar{b} \rangle
  \]
  (where, for a set $A$, $\overline{a} \in A$ means $a_i \in A$ for each $i \leqslant \operatorname{length}(\overline{a})$).
  Comparing this with Definition~\ref{def:ident_si} , we see that $h(\bar{\mu}) \in \mathbb{C}(\bar{\mu})$ is \emph{identifiable} if and only if it is $\mathcal{L}_{DF, \mathbb{C}}$-\emph{definable} over $(\bar{y}^\ast, \bar{u}^\ast)$ for every generic solution $(\bar{x}^\ast, \bar{y}^\ast, \bar{u}^\ast)$ of ~\eqref{eq:sigma}.
\end{example}

\begin{remark}\label{rem:definability}
  For a saturated model $M$, definability can be restated in terms of automorphisms~\cite[Proposition 4.3.25]{Marker_book}: for $\bar{a}, \bar{b} \in M$, $\bar{a}$ is definable over $\bar{b}$ if and only if
  \[
    \forall \text{ automorphism } \alpha \colon M \to M \quad \alpha(\bar{b}) = \bar{b} \implies \alpha(\bar{a}) = \bar{a}.
  \]
  Informally, this can be stated as \emph{if $\bar{b}$ is fixed, then $\bar{a}$ is also fixed}.
  Syntactically, this is very similar to the analytic definition of identifiability~\cite[Definition~2.5]{HOPY2020}.
  This partially explains why model theoretic tools were used in proving the equivalence~\cite[Proposition~3.4]{HOPY2020} of the analytic definition and Definition~\ref{def:ident_si}.
\end{remark}

In order to define multi-experiment identifiability in model-theoretic terms, we will define the notion of independence.

\begin{our_definition}[Stationarity, nonforking, and independence]\label{def:independence}
  Let $k$ be a differential subfield of a differentially closed field $K$ and let $n$ be a positive integer.
  Let $\bar{a}$ be $n$-tuple of elements from $K$ and $\bar{x}$ denote $n$-tuple of differential variables.
  \begin{itemize}
  \item Recall from Notation~\ref{not:tuples} that the vanishing ideal $\mathcal{I}_{k}({\bar a})$ of ${\bar a}$ over $k$ is $\{P\in k\{\bar x\}: P(\bar a) = 0\}$.   
  Note that $\mathcal{I}_{k}({\bar a})$ depends only on $p := \tp({\bar a}/k)$.
  Moreover, by quantifier elimination of $\DCF_{0}$ (see Remark~\ref{rem:qe}), it also determines $\tp({\bar a}/k)$, so we may write it as $\mathcal{I}(p)$.
  \item Let $L$ be a differential field with $k \subset L \subset K$.  
  We say that ${\bar a}$ is \emph{independent from $L$ over $k$} if $\mathcal{I}_{L}({\bar a})$ is a prime component of $\mathcal{I}_{k}({\bar a})\otimes_{k}L$.  
  We also express this by saying that $\tp({\bar a}/L)$ \emph{does not fork over $k$}, or that $\tp({\bar a}/L)$ is a \emph{nonforking extension of $\tp({\bar a}/k)$}.
  \item We say that $\tp({\bar a}/k)$ is \emph{stationary} if  $\mathcal{I}_{k}({\bar a})$ is ``absolutely prime", namely for each $L \supset k$, $\mathcal{I}_{k}({\bar a})\otimes_{k}L$ is prime. 
  It is enough to require this for $L = K$. 
   \item If ${\bar b}$ is another finite tuple from $K$, we say that ${\bar a}$ and ${\bar b}$ are \emph{independent} over $k$ if ${\bar b}$ is independent from $k\langle {\bar a}\rangle$ over $k$.
   A sequence of tuples $\bar{a}_1, \bar{a}_2, \ldots$ is called \emph{independent over $k$} if, for every $i \geqslant 1$, $a_{i+1}$ is independent from $k\langle{\bar a}_{1},...,{\bar a}_{i}\rangle$ over $k$. 
   \item In general, given subsets $A\subseteq B$ of $K$, we say that ${\bar a}$ is \emph{independent from $B$ over $A$}, or \emph{$\tp({\bar a}/B)$ does not fork over $A$}, if ${\bar a}$ is independent from $L_{2}$ over $L_{1}$ where $L_{1}$ and $L_2$ are the differential fields generated by $A$ and $B$, respectively.
\end{itemize}
\end{our_definition}

\begin{example}[Generic solution of~\eqref{eq:sigma} is stationary]\label{ex:stationary}
   Let $\bar{a} := (\bar{x}^\ast, \bar{y}^\ast, \bar{u}^\ast)$ be a generic solution of~\eqref{eq:sigma} (see Definition~\ref{def:generic_solution}).
   Then, by the definition, $\I_{\mathbb{C}(\bar{\mu})}(\bar{a}) = I_{\Sigma}$.
   \cite[Proof of Lemma~3.2]{HOPY2020} implies that the ideal $I_\Sigma$ is prime and remains prime under any field extension.
   Therefore, $\tp(\bar{a} / \mathbb{C}(\bar{\mu}))$ is stationary.
\end{example}

\begin{remark}[Some properties of independence and forking]\label{rem:some_prop}
  In this remark, we use the notation from Definition~\ref{def:independence}.
  \begin{enumerate}[label=(\arabic*)]
      \item 
      One can show that finite tuples ${\bar a}$, ${\bar b}$ are independent over $k$ if and only if the ideal $\mathcal{I}_{k}({\bar a}, {\bar b}) \subseteq k\{{\bar x}, {\bar y}\}$ is a prime component of the ideal $I$ of $k\{{\bar x}, {\bar y}\}$ generated by $\mathcal{I}_{k}({\bar a})$ and  $\mathcal{I}_{k}({\bar b})$.   
      Moreover, if both $\tp({\bar a}/k)$ and $\tp({\bar b}/k)$ are stationary, $I$ is itself prime.
      \item\label{rem715item2} Using the fact that algebraic (in)dependence is invariant under  extension of scalars, one can show that ${\bar a}$ is independent from $L \supset k$ over $k$ if and only if, for every $m$, we have
      \[
      \trdeg_{k} k\left({\bar a}, {\bar a}',\ldots,{\bar a}^{(m)}\right) = \trdeg_{L} L\left({\bar a}, {\bar a}',\ldots,{\bar a}^{(m)}\right).
      \]
      Together with~\cite[Proposition~1.16]{Anand_notes}, this implies that the definition of independence in $\DCF_0$ we gave agrees with the general model-theoretic one (as e.g., in~\cite[\S 2.2,  page~28]{Pillay_book}).
      \item\label{rem715item3} The definition of stationarity implies that, for every stationary type $p$ over $k$ and every differential field $L \supset k$, there is a unique complete type $q$ that extends $p$ and that does not fork over $k$. Such a type $q$
      will be referred to as \emph{the} nonforking extension of $p$. Note that  the type $q$ is again stationary.
      One can show that the converse (the uniqueness of nonforking extension of $p$ implies the stationarity of $p$) is also true by using the characterization of independence from the previous item and the fact that, after the extension of scalars, an irreducible variety becomes equidimensional.
      
  \end{enumerate}
\end{remark}

\begin{example}
  Consider the differential field $k = \Q(t)$ with respect to the derivation $\frac{\operatorname{d}}{\operatorname{dt}}$ and a saturated model $K \supset k$ of $\DCF_0$.
  Every formula in the type $p := \tp(t / \Q)$  is implied by the
  single formula $x' = 1$.
  Then type $q := \tp(t / \Q(t))$ is an extension of $p$, and it contains a new formula $x = t$, which is not implied by $x'=1$.
  We have
  \[
    \trdeg_{\Q} t = 1 \neq 0 = \trdeg_{\Q(t)} t,
  \]
  so the extension of $p$ by $q$ is forking. 
  
  Also, from the differential equations theory, we know that
  the general solution of $x' = 1$ is of the form $x = t + c$, where $c$ is a constant.
  So we can construct a nonforking extension of $p$ to $\Q(t)$ as $\tp((t + c) / \Q(t))$, where $c \in K$ is a transcendental constant (exists because $K$ is saturated). 
\end{example}

\begin{example}[Multi-experiment identifiability via independence]\label{ex:ident_mult}
   Let $\bar{a} = (\bar{x}^\ast_1, \bar{y}^\ast_1, \bar{u}^\ast_1)$ and $\bar{b} = (\bar{x}^\ast_2, \bar{y}^\ast_2, \bar{u}^\ast_2)$,
   where $(\bar{x}^\ast_1, \bar{x}^\ast_2, \bar{y}^\ast_1, \bar{y}^\ast_2, \bar{u}^\ast_1, \bar{u}^\ast_2)$ is a generic solution of $\Sigma_2$ (see Definition~\ref{def:ident:multi}).
   By the definition of $I_{\Sigma_2}$, it is generated by two copies of $I_\Sigma$, so $\bar{a}$ and $\bar{b}$ are independent over $\mathbb{C}(\bar{\mu})$.
   Moreover, $\bar{a}$ and $\bar{b}$ are \emph{independent realizations} of the type of a generic solution of~\eqref{eq:sigma} (see Example~\ref{ex:type_generic}).
   Combining this with Example~\ref{ex:ident_si}, we have that
   \begin{displayquote}
   \emph{$h(\bar{\mu}) \in \mathbb{C}(\bar{\mu})$ is multi-experiment identifiable if and only if it is definable in $\mathcal{L}_{DF, \mathbb{C}}$ over some finite number of independent realizations of $\tp((\bar{y}^\ast, \bar{u}^\ast) / \C(\bar{\mu}))$, where $(\bar{x}^\ast, \bar{y}^\ast, \bar{u}^\ast)$ is a generic solution of~\eqref{eq:sigma}.}
   \end{displayquote}
\end{example}

Finally, it has been shown in~\cite[Theorem~19]{allident} that the field of multi-experiment identifiable functions coincides with the field of definition of the ideal of input-output relations.
Any set of generators of the field of definition is called a canonical base in model theory:

\begin{our_definition}[{Canonical base, \cite[Definition~8.2.2]{Marker_book}}]\label{def:canonical}
  Let $M$ be a saturated model of the theory $\DCF_0$ (that is, large enough differentially closed field, see Definition~\ref{def:saturated}), and $p$ be a complete type over $M$.
  Then a set $A \subset M$ is called \emph{a canonical base}  of $p$ if and only if
  \[
  \forall \text{ automorphism } \alpha \colon M \to M \quad \alpha(p) = p \iff \alpha|_A = \operatorname{id},
  \]
  where the automorphism acts on the type by acting on the formulas defining the type (which are defined over $M$, see also Example~\ref{ex:types}).
  In particular, an automorphism fixes a complete type if it leaves the corresponding set of formulas invariant.
  
   Every canonical base of a complete type $p$ generates the same differential field over $k$~\cite[Lemma~8.2.4]{Marker_book}.
  This field will be denoted by $\cb(p)$ and referred to as \emph{the canonical base} (see~\cite[p. 29]{Pillay_book}).
  If $k$ is a differential subfield of a differentially closed $K$ and $\bar{a}$ is a tuple from $K$ such that $\tp(\bar{a} / k)$ is stationary, then  $\cb(\bar{a} / k)$  denotes the canonical base of the nonforking extension of $\tp(a / k)$ to $K$ (see Definition~\ref{def:independence}).
\end{our_definition}

\begin{example}\label{ex:ident_me}
  In the theory of differential fields, the canonical base of stationary $\tp(\bar{a} / k)$ is the field of definition of $\mathcal{I}_k(\bar{a})$.
  Therefore,
  \cite[Theorem~19]{allident} can be rephrased as follows:
  \begin{displayquote}
  \emph{the field of multi-experimental identifiable functions  is $\cb((\bar{y}^\ast, \bar{u}^\ast) / \C(\bar{\mu}))$, where $(\bar{x}^\ast, \bar{y}^\ast, \bar{u}^\ast)$ is any generic solution of~\eqref{eq:sigma}.}
  \end{displayquote}
  This fact sounds natural if one looks at Definition~\ref{def:canonical}: the canonical base is fixed if and only if the set of all experimental outcomes for fixed generic parameters is invariant.
\end{example}

We will conclude this subsection by summarizing some properties of forking extensions which will be used in the subsequent proofs.
\begin{remark}[Properties of forking]\label{rem:forking}
We fix a differentially closed field $K \supset k_0$, its subsets $A \subseteq B \subseteq C$, and a tuple $\bar{a}$ from $K$.
\begin{enumerate}[label=(\arabic*)]
    \item \emph{(transitivity, \cite[Proposition 2.20(iii)]{Pillay_book})} $\tp(\bar{a} / C)$ does not fork over $A$ if and only if it does not fork over $B$ and $\tp(\bar{a} / B)$ does not fork over $A$.
    
    \item\label{rem720item2} \emph{(symmetry, \cite[Proposition~2.20(v)]{Pillay_book})} $\tp(\bar{a} / B)$ does not fork over $A$ if and only if, for every $\bar{b}$ from $B$, $\tp(\bar{b} / A \cup \bar{a})$ does not fork over $A$.
    
    \item\label{rem720item3} \emph{(\cite[Proposition~2.20(iv)]{Pillay_book})}
    Assume that  two distinct types $\tp(\bar{a}/B)$ and $\tp(\bar{b}/B)$ do not fork over $A$.
    Assume also that $\tp(\bar{a} / A) = \tp(\bar{b} / A)$.
    Then there exists an equivalence relation $E(x_1, x_2)$ defined over $A$ with finitely many classes such that, for every $\bar{a}^\ast$  satisfying $\tp(\bar{a} / B)$ and $\bar{b}^\ast$  satisfying $\tp(\bar{b} / B)$, we have $\neg E(\bar{a}^\ast, \bar{b}^\ast)$.
    In geometric terms, one can think of $E$ being the relation ``belong to the same component of the variety defined by $\tp(\bar{a} / A)$''.
   
    \item\label{rem720item4} \emph{(forking and canonical bases, \cite[Remark~2.26]{Pillay_book})} If $p$ is a stationary type over $A$, and let $F$ be the differential fields generated by $A$.
    Then $\cb(p)\subseteq F$ and coincides with $\cb(q)$ whenever $q$ is the nonforking extension of $p$ to a larger set $B\supseteq A$. 
    Also $\tp({\bar a}/B)$ does not fork over $A\subseteq B$ iff $\cb(\tp({\bar a}/B)$ is contained in the algebraic closure of $F$. 
\end{enumerate}
\end{remark}


  \subsection{Single-output model requiring more than one experiment revisited}\label{sec:cemodeltheory}
  
 In this section, we will discuss a model-theoretic construction used to find~\eqref{eq:counterexsys}.
We will work over the field $\C$, that is, in the language $\mathcal{L}_{DF, \mathbb{C}}$.
Consider a constant differential field $k = \C(\bar{\mu})$.
Let $K$ be a saturated differentially closed field containing $k$.
Let $\mathcal{C}$ denote the constants of $K$.
We will use two technical lemmas.

\begin{lemma}\label{lem:cb_over_c}
  For every tuple $\bar{a}$ from $K$, $\cb(\bar{a} / \mathcal{C}) = \mathcal{C} \cap \C\langle \bar{a} \rangle$.
\end{lemma}

\begin{proof}
  First we observe that, since $\mathcal{C}$ is algebraically closed~\cite[Lemma~2.1]{MarkerMTDF}, the type $\tp(\bar{a} / \mathcal{C})$ is stationary due to~\cite[Chapter~1, Remark~2.25(i)]{Pillay_book}, so we can use Definition~\ref{def:canonical}.
  Consider any automorphism $\alpha$ of $K$ such that $\alpha(\bar{a}) = \bar{a}$.
  Since $\mathcal{C}$ is the field of constants of $K$, we have  $\alpha(\mathcal{C}) = \mathcal{C}$.
  Then $\alpha$ fixes $I_{\mathcal{C}}(\bar{a})$ setwise, so it fixes the nonforking extension of $\tp(\bar{a} / \mathcal{C})$ to $K$ (see Definition~\ref{def:independence}).
  By Definition~\ref{def:canonical}, we conclude that $\alpha$ fixes $\cb(\bar{a} / \mathcal{C})$.
  Thus, Remark~\ref{rem:definability} implies that $\cb(\bar{a} / \mathcal{C}) \subset \C\langle\bar{a} \rangle$.
  
  In the other direction, consider $b \in \C\langle \bar{a} \rangle \cap \mathcal{C}$.
  There exists a differential rational function $f$ over $\C$ such that $b = f(\bar{a})$.
  Therefore, the formula $b = f(\bar{x})$ belongs to $\tp(\bar{a} / \mathcal{C})$, so it belongs to its nonforking extension $p$ to $K$.
  Then any automorphism $\alpha$ of $K$ fixing $p$ fixes $\bar{b}$.
  Then Definition~\ref{def:canonical} implies that $b \in \cb(\bar{a} / \mathcal{C})$.
\end{proof}

\begin{lemma}\label{lem:cb_inheritance}
  Let $\bar{a}$ be a tuple from $K$ and $\bar{c}$ any set of generators of $\cb(\bar{a} / \mathcal{C})$  as a field over $\mathbb{C}$.
  Then
  \begin{enumerate}
      \item $\cb(\bar{a} / k) = \cb(\bar{c} / k)$;
      \item $\C\langle \bar{a} \rangle \cap k = \C(\bar{c}) \cap k$.
  \end{enumerate}
\end{lemma}

\begin{proof}
  \begin{enumerate}
      \item[]
      \item Lemma~\ref{lem:cb_over_c} implies that $\bar{c}$ is a tuple from $\C\langle \bar{a} \rangle$.
      Therefore, using Remark~\ref{rem:definability}, we obtain $\cb(\bar{c} / k) \subset \cb(\bar{a} / k)$.
      
      Consider an automorphism $\alpha$ of $K$ that fixes the nonforking extension $p$ of $\tp(\bar{c} / k)$ to $K$.
      Then $\alpha$ fixes $K \otimes_k \mathcal{I}_k(\bar{c})$ setwise, so
      $\tp(\bar{c} / k) = \tp(\alpha(\bar{c}) / k)$. Then Remark~\ref{rem:types_mapping} implies that there exists an automorphism $\beta$ of $K$ that fixes $k$ and  $\beta(\alpha(\bar{c})) = \bar{c}$.
      Then $\tp(\bar{a} / \mathcal{C}) = \tp(\beta(\alpha(\bar{a})) / \mathcal{C})$, so $\tp(\bar{a} / k) = \tp(\beta(\alpha(\bar{a})) / k)$.
      Since $\beta^{-1}$ fixes $k$, we have
      \[
        \tp(\bar{a} / k) = \tp(\beta(\alpha(\bar{a})) / k) = \tp(\beta^{-1}(\beta(\alpha(\bar{a}))) / k) = \tp(\alpha(\bar{a}) / k).
      \]
      Therefore, $\mathcal{I}_k(\bar{a}) = \mathcal{I}_k(\alpha(\bar{a}))$, so these types have the same nonforking extensions to $K$.
      Hence $\alpha$ fixes $\cb(\bar{a} / k)$. 
      Thus, $\cb(\bar{a} / k) \subset \cb(\bar{c} / k)$.
      
      \item Using Lemma~\ref{lem:cb_over_c}, since $k$ is constant, we have \[\C\langle \bar{a} \rangle \cap k = (\C\langle \bar{a} \rangle \cap \mathcal{C}) \cap k = \C(\bar{c}) \cap k.\qedhere\]
  \end{enumerate}
\end{proof}

Let $(\bar{x}^\ast, \bar{y}^\ast, \bar{u}^\ast)$ be a generic solution of an algebraic differential model $\Sigma$ as in~\eqref{eq:sigma}.
Then the desired non-equality
of the fields of SE- and ME-identifiable functions can be restated, using Examples~\ref{ex:ident_si} and~\ref{ex:ident_me}, as
\begin{equation}\label{eq:dcl_not_cb}
    \C\langle \bar{y}^\ast, \bar{u}^\ast \rangle \cap \C (\bar{\mu}) \neq \cb((\bar{y}^\ast, \bar{u}^\ast) / \C (\bar{\mu})).
\end{equation}
We will first 
 construct an example having constant dynamics and satisfying the non-equality~\eqref{eq:dcl_not_cb}
(as, for example, in~\cite[Example~2.14]{HOPY2020}). We then use Lemma~\ref{lem:cb_inheritance} to ``pack'' two output variables  of the example into a single 
 output variable while preserving~\eqref{eq:dcl_not_cb}:
\begin{enumerate}[label = (Step~\arabic*), leftmargin=15mm]
    \item Let $\bar{\mu} = (\mu_1, \mu_2)$, and introduce a constant state variable $x_1$.
    We introduce two auxiliary outputs $z_1 = x_1$ and $z_2 = \mu_1 x_1 + \mu_2$.
    The defining differential ideal of $(z_1, z_2)$ is generated by $z_1', z_2 - \mu_1 z_1 - \mu_2$, so $\cb((z_1, z_2) / \C(\bar{\mu})) = \C(\bar{\mu})$.
    On the other hand, since the automorphism of $\C(x_1, \mu_1, \mu_2)$ defined by 
    \[
      x_1 \to x_1,\ \ \mu_1 \to \mu_1 + 1,\ \ \mu_2 \to \mu_2 - x_1
    \] 
    fixes $z_1$ and $z_2$ but does not fix $\mu_1$ or
    $\mu_2$, we conclude that $\mu_1, \mu_2 \not\in \C(z_1, z_2)$.
    
    \item Now we introduce a new state variable $x_2$ satisfying $x_2' = z_1x_2 + z_2$ and set the output $y = x_2$.
    Then we have $\cb(y / \mathcal{C}) = \C(z_1, z_2)$.
    Therefore, using Lemma~\ref{lem:cb_inheritance},
    \[
      \cb(y / \C(\bar{\mu})) = \cb((z_1, z_2) / \C(\bar{\mu})) = \C(\bar{\mu}) \neq \C(z_1, z_2) \cap \C(\bar{\mu}) = \C\langle y \rangle \cap \C(\bar{\mu}).
    \]
    So we get exactly~\eqref{eq:counterexsys}.
\end{enumerate}

\begin{remark}
  Instead of $x_2' = z_1x_2 + z_2$, we could take any other equation containing $z_1$ and $z_2$ among the coefficients, for example, $x_2' = z_1x_2^2 + 2z_2x_2- 3z_1^3$.
  This would yield an example with the same property~\eqref{eq:dcl_not_cb}.
\end{remark}

\begin{remark}
  Another way to obtain an example with the property from Conclusion~\ref{thm:example} is to remove the first output ($x_0$) from any of the models in Example~\ref{ex:lincomp}.
\end{remark}


\subsection{Model theory way of proving  Theorem~\ref{thm:num_exp}}\label{sec:mtpf}

In this section, we will use Notation~\ref{not:tuples} and~\ref{not:fields}. 
In particular, we work \emph{over a fixed ground differential field $k_0$} (that is, in $\mathcal{L}_{DF, k_0}$).
Therefore, all fields are assumed to be generated over $k_0$ and, in particular, whenever we write $\tp(\bar{a} / A)$, this is equivalent to $\tp(\bar{a} / k_0 \langle A \rangle)$.

\begin{proposition}[{Model theoretic reformulation of Proposition~\ref{prop:existence_N}}]\label{prop:existence_N_mt}
  Let $k \subset K$ be a differential subfield, $p$ a stationary type over $k$, and a sequence $\bar{a}_1, \bar{a}_2, \ldots$ of independent realizations of $p$ in $K$.
  Then there exists $N$ such that 
  \[
    \cb(p) \subset k_0\langle \bar{a}_1, \ldots, \bar{a}_N \rangle.
  \]
\end{proposition}

\begin{proof}
  This follows from~\cite[Chapter~1, Lemma~3.19]{Pillay_book} (see also~\cite[Exercise~8.4.12]{Marker_book}).
\end{proof}

\begin{lemma}[{cf. Lemma~\ref{lem:independence}}]\label{lem:indep_mt}
  Let $k_0 \subset k \subset K$ be a differential field.
  Consider tuples $\bar{a}, \bar{b}$ from $K$ such that 
  $\tp(\bar{a} / k)$ is stationary and $\bar{b}$ is independent from $\bar{a}$ over $k$.
  Then $\cb(\bar{a} / k) = \cb(\bar{a} / k\langle \bar{b} \rangle)$.
\end{lemma}

\begin{proof}
  Let $p_1 := \tp(\bar{a} / k)$ and $p_2 := \tp(\bar{a} / k\langle \bar{b} \rangle)$.
  Consider a subfield $F \subset k$. Then
  \begin{itemize}
     \item the independence of $\bar{a}$ and $\bar{b}$ together with Remark~\ref{rem:forking}\ref{rem720item3} implies that $p_1$ does not fork over $F$ iff $p_2$ does not fork over $F$;
     \item since $p_1$ is the restriction of $p_2$ to $k$, $p_1$ restricted to $F$ is stationary iff $p_2$ restricted to $F$ stationary.
  \end{itemize}
  Applying Remark~\ref{rem:forking}\ref{rem720item4} twice, with $A = \cb(p_1)$ and with $A = \cb(p_2)$, we conclude $\cb(p_1) = \cb(p_2)$.
\end{proof}

\begin{proposition}[{Model theoretic version of Proposition~\ref{prop:bound_N}}]
  Let:
  \begin{itemize}
  \item $k$ with $k_0 \subset k \subset K$ be a differential field of finite transcendence degree over $k_0$,
  \item 
  $p$ a stationary type over $k$ (in particular, $p$ is a complete type),
  \item 
  a sequence $\bar{a}_1, \bar{a}_2, \ldots$ of independent realizations of $p$ in $K$,
  \item 
  $r$ the smallest integer such that 
  \[
  \trdeg_{k_0\langle \bar{a}_1, \ldots, \bar{a}_{r} \rangle} k = \trdeg_{k_0\langle \bar{a}_1, \ldots, \bar{a}_{r + 1} \rangle} k.
  \]
  \end{itemize}
  Then 
  \begin{enumerate}[label=(\arabic*)]
      \item $r$ is the smallest integer such that $\cb(p)$ is algebraic over $k_0\langle \bar{a}_1, \ldots, \bar{a}_r \rangle$;
      \item $\cb(p) \subset k_0\langle \bar{a}_1, \ldots, \bar{a}_{r + 1}\rangle$.
  \end{enumerate}
\end{proposition}

\begin{proof}
  Let $\bar{c}$ be any finite tuple of generators of $\cb(\bar{a}_1 / k)$.
  For every nonnegative integer $\ell$, we set $\bar{A}_\ell := (\bar{a}_1, \ldots, \bar{a}_\ell)$
  We consider the following sequence of types (cf. the varieties $X_i$ in the proof of Proposition~\ref{prop:bound_N}):
  \[
  \tp(\bar{c} / \bar{A}_0) \subset \tp(\bar{c} / \bar{A}_1) \subset \tp(\bar{c} / \bar{A}_2) \subset \ldots
  \]
  \textbf{Claim:} \emph{$\bar{c}$ is algebraic over $k_0\langle \bar{A}_{s}\rangle$ iff $\trdeg_{k_0\langle \bar{A}_s\rangle} \bar{c} = \trdeg_{k_0\langle \bar{A}_{s + 1}\rangle} \bar{c}$.}
  If the equality of the transcendence degrees does not hold, $\trdeg_{k_0\langle \bar{A}_s\rangle} \bar{c} > 0$, so $\bar{c}$ is not algebraic over this field.
  
  We will now show the reverse implication.
  The equality of transcendence degrees implies
  that $\tp(\bar{c} / k_0\langle \bar{A}_{s + 1}\rangle)$ does not fork
  over $k_0\langle \bar{A}_{s}\rangle$.
  The symmetry of forking (Remark~\ref{rem:forking}\ref{rem720item2}) implies that $\tp(\bar{a}_{s + 1} / (\bar{A}_s, \bar{c}))$ does not fork over  $\bar{A}_{s}$.
  Therefore, by Remark~\ref{rem:some_prop}\ref{rem715item2}, we have
  \begin{equation}\label{eq:trdeg_eq}
    \trdeg_{k_0\langle \bar{A}_s\rangle} \bar{a}_{s + 1} = \trdeg_{k_0\langle \bar{c}, \bar{A}_{s}\rangle} \bar{a}_{s + 1}.
  \end{equation}
  Since $\bar{a}_{1}, \ldots, \bar{a}_{s + 1}$ are independent, $\tp(\bar{a}_{s + 1} / k_0\langle \bar{c}, \bar{A}_s \rangle)$ does not fork over $k_0\langle \bar{c} \rangle$.
  Therefore, Remark~\ref{rem:forking}\ref{rem720item4}
  (applying with $A = k_0\langle \bar{c} \rangle$) implies that
  \[
  \cb(\bar{a}_{s + 1} / (\bar{A}_s, \bar{c})) = k_0\langle \bar{c} \rangle.
  \]
  On the other hand, Remark~\ref{rem:forking}\ref{rem720item4} applied with $A = k_0\langle \bar{A}_s\rangle$ together with~\eqref{eq:trdeg_eq} imply that $\cb(\bar{a}_{s + 1} / (\bar{A}_s, \bar{c}))$ is algebraic over $k_0\langle \bar{A}_s \rangle$.
  So, the claim is proved.
  
  Fix $s \geqslant 1$.
  Since $\bar{c}$ generates $\cb(\bar{A}_s / k)$, then Remark~\ref{rem:forking}\ref{rem720item4} implies that $\tp(\bar{A}_s / k)$ is a nonforking extension of $\tp(\bar{A}_s / k_0\langle \bar{c} \rangle)$.
  Let $\bar{d}$ be any finite tuple of generators of $k$ over $k_0$.
  The symmetry of forking (Remark~\ref{rem:forking}\ref{rem720item2}) implies that $\tp(\bar{d} / k_0\langle \bar{c}, \bar{A}_s \rangle )$ does not fork over $k_0\langle \bar{c} \rangle$.
  Thus, 
  \[
  \trdeg_{k_0\langle \bar{A}_s \rangle} k = \trdeg_{k_0\langle \bar{A}_s \rangle} \bar{c} + \trdeg_{k_0\langle \bar{c}, \bar{A}_s \rangle} k =  \trdeg_{k_0\langle \bar{A}_s \rangle} \bar{c} + \trdeg_{k_0\langle \bar{c} \rangle} k.
  \]
  Therefore, (cf. Lemma~\ref{lem:trdeg} and~\eqref{eq:trdeg_decomposition})
  \[
    \trdeg_{k_0\langle \bar{A}_s\rangle} \bar{c} = \trdeg_{k_0\langle \bar{A}_{s + 1}\rangle} \bar{c} \iff \trdeg_{k_0\langle \bar{A}_s\rangle} k = \trdeg_{k_0\langle \bar{A}_{s + 1}\rangle} k.
  \]
  Together with the claim, this proves the first statement of the proposition.
 
Now we prove the second part of the proposition.
  We will do this by showing that 
  \begin{enumerate}[label=(P\arabic*)]
      \item\label{step:nonfork} $\tp(\bar{a}_{r + 2} / k\langle \bar{A}_{r + 1}\rangle)$ does not fork over $k_0\langle \bar{A}_{r + 1} \rangle$ and
      \item\label{step:stat}  $\tp(\bar{a}_{r + 2} / k_0\langle \bar{A}_{r + 1}\rangle)$ is stationary.
  \end{enumerate}
  If we prove these two statements, then, by Remark~\ref{rem:forking}\ref{rem720item4}, we will have 
  \[
  \cb\left(\bar{a}_{r + 2} / k\langle \bar{A}_{r + 1}\rangle\right) \subset k_0\langle \bar{A}_{r + 1} \rangle.
  \]
  Then Lemma \ref{lem:indep_mt} 
   will
  imply that $\cb(\bar{a}_{r + 2} / k) = \cb\left(\bar{a}_{r + 2} / k\langle \bar{A}_{r + 1}\rangle\right) \subset k_0 \langle \bar{A}_{r + 1}\rangle$.
  
  By the first part of the proposition, $\cb(\bar{a}_{r + 2} / k)$ is algebraic over $k_0\langle \bar{A}_{r} \rangle \subset k_0\langle \bar{A}_{r + 1} \rangle$.
  Lemma~\ref{lem:indep_mt} implies that $\cb(\bar{a}_{r + 2} / k) = \cb(\bar{a}_{r + 2} / k\langle \bar{A}_{r + 1}\rangle)$.
  Then Remark~\ref{rem:forking}\ref{rem720item4} applied to type $\tp(\bar{a}_{r + 2} / k\langle \bar{A}_{r + 1}\rangle)$ and $A = k_0\langle \bar{A}_{r + 1} \rangle$ implies that the type does not fork over $A$.
  This proves~\ref{step:nonfork}.
  
  It remains to prove~\ref{step:stat}.
  Assume the contrary, that is, $\tp(\bar{a}_{r + 2} / k_0\langle \bar{A}_{r + 1}\rangle)$
  has at least two nonforking extensions to $k\langle \bar{A}_{r + 1}\rangle$.
  One of them is $p_0 := \tp(\bar{a}_{r + 2} / k\langle \bar{A}_{r + 1}\rangle)$. 
  Since $\bar{a}_{r + 2}$ and $\bar{A}_{r + 1}$ are independent over $k$ $p_0$ does not fork over $k$ (see Definition~\ref{def:independence}).
  Since the type $\tp(\bar{a}_{r + 2} / k)$ is stationary,
Remark~\ref{rem:some_prop}\ref{rem715item3} implies that the type $p_0$ is stationary.
  We denote an extension different from $p_0$ by $q_0 := \tp(\bar{b} / k\langle \bar{A}_{r + 1}\rangle)$.
  Since $\tp(\bar{a}_{r + 2} / k\langle \bar{A}_r \rangle)$ is stationary (similarly to $p_0$), the restrictions $p_1$ and $q_1$ of $p_0$ and $q_0$, respectively, to $k\langle \bar{A}_r \rangle$ are distinct.
  Moreover, since $p_0$ does not fork over $k_0\langle \bar{A}_r \rangle$, the same is true for $q_0$.
  Therefore, Remark~\ref{rem:forking}\ref{rem720item3}, applied to $p_0$ and $q_0$ as distinct nonforking extensions of $\tp(\bar{a}_{r + 1} / k_0 \langle \bar{A}_r \rangle)$ yields a finite equivalence relation $E$ defined over $k_0\langle \bar{A}_r\rangle$ such that $\neg E(\bar{a}_{r + 2}, \bar{b})$.
  Since $\bar{b}$ and $\bar{a}_{r + 2}$ are of the same type over $k_0\langle \bar{A}_{r + 1}\rangle$, we have $\neg E(\bar{a}_{r + 1}, \bar{a}_{r + 2})$.
  Since  the type $p$ is stationary and $\bar{a}_{r + 1}, \bar{a}_{r + 2}, \ldots$ are independent over $k\langle A_r\rangle$, we have (cf.~\cite[Lemma~2.28]{Pillay_book}) that, for every $i \neq j$ such that $i, j > r$,
\[
  \tp((\bar{a}_i, \bar{a}_j) / k\langle A_r\rangle) = \tp((\bar{a}_{r + 1}, \bar{a}_{r + 2}) / k\langle A_r\rangle).
\]
  Therefore, we have $\neg E(\bar{a}_i, \bar{a}_j)$. This contradicts the fact that $E$ defines only finitely many equivalence classes.
  The contradiction finishes the proof of~\ref{step:stat}.
\end{proof}

\section*{Acknowledgments}
We are grateful to Julio Banga and Alejandro Villaverde for learning from them about the importance of  multi-experiment parameter identifiability at the AIM workshop ``Identifiability problems in systems biology'', to Alexandre Sedoglavic for helpful discussion about his algorithm~\cite{Sedoglavic}, and to the referees for their suggestions.
This work was partially supported by the NSF grants  CCF-1564132 and 1563942, DMS-1760448, 1760413, 1853650, 1665035, 1760212, 1853482, 1800492, and 2054721, by the Paris Ile-de-France Region, and by the Fields Institute for Research in Mathematical Sciences

\setlength{\bibsep}{5pt}
\bibliographystyle{abbrvnat}
\bibliography{bib}

\end{document}